\documentclass[12pt]{article}
\usepackage{graphicx} 

\usepackage{float}

\usepackage{amsfonts}
\usepackage{amsthm}
\usepackage{amsmath}
\usepackage{amssymb}

\usepackage{comment}
\usepackage{xcolor}

\usepackage{tikz}
\usetikzlibrary{arrows.meta, decorations.pathreplacing}

\usepackage{marginnote}

\newcommand{\lG}{\mathcal{G}}
\newcommand{\lH}{\mathcal{H}}

\newcommand{\lV}{\mathcal{V}}

\newcommand{\dd}{\mathrm{d}}

\theoremstyle{plain} 
\newtheorem{theorem}{Theorem}[section]
\newtheorem{mtheorem}{Theorem}
\newtheorem{lemma}[theorem]{Lemma}

\newtheorem{proposition}[theorem]{Proposition}

\theoremstyle{definition}
\newtheorem{definition}{Definition}

\newtheorem{question}{Question}

\theoremstyle{remark}

\title{The Uniform Random Walk on graphs, loop processes and graphings}
\author{Miklos Abert, Adam Arras, Jaelin Kim}
\date{May 2025}

\begin{document}

\maketitle

\begin{abstract}
We define the \emph{Uniform Random Walk} (URW) on a connected, locally finite graph as the weak limit of the uniform walk of length $n$ starting at a fixed vertex. When the limit exists, it is necessarily Markovian and is independent of the starting point. For a finite graph, URW equals the Maximal Entropy Random Walk (MERW). 

We investigate the existence and phase transitions of URW for loop perturbed regular graphs and their limits. It turns out that for a sequence of finite graphs, it is the global spectral theory of the limiting graphing that governs the behavior of the finite MERWs. 

In the delocalized phase, we use a `membrane argument', showing that the principal eigenfunction of an expander graphing is stable under a small diagonal perturbation. This gives us: 1) The existence of URW on leaves; 2) The URW is a unique entropy maximizer; 3) The MERW of a finite graph sequence Benjamini-Schramm converges to the URW of the limiting graphing. 

In the localized phase, the environment seen by the particle takes the role of a finite stationary measure. We show that for canopy trees, the URW exists, is transient and maximizes entropy. We also show that for large finite graphs where most vertices have a fixed degree, localization of MERW is governed by the adjacency norm. 

\end{abstract}

\newpage

\section{Introduction}

Let $G$ be a locally finite, connected graph. We aim to understand which
random walk (that is, nearest neighbor Markov chain) on $G$ has maximal
entropy. Simple Random Walk (SRW) maximizes stepwise entropy, but already
for finite non-regular graphs, it is globally suboptimal. In the finite
case, the space of walk trajectories $X_{G}$ is a subshift of finite type,
and random walks on $G$ induce shift invariant probability measures on it.
Under standard assumptions on $G$, $X_{G}$ admits a unique measure, whose
Kolmogorov-Sinai (KS) entropy equals the topological entropy of $X_{G}$ and
is therefore maximal~\cite{goodman1971relating}. This measure was introduced
in~\cite{parry1964intrinsic} and the associated random walk has been
extensively studied for finite graphs in both theoretical~\cite%
{bowen1970markov,walters1978equilibrium,lind1995symbolic,burda2009localization}
and applied contexts~\cite%
{brin1998anatomy,langville2006pagerank,ochab2013maximal} under the name
Maximal Entropy Random Walk (MERW).

The theory for\emph{\ }infinite\emph{\ }graphs is much less understood,
starting with the question how to measure entropy. When insisting to use KS 
entropy, one is essentially confined to positive recurrent walks. This has
been developed in the early works of Vera-Jones and Gurevich \cite%
{vere1967ergodic,vere1968ergodic,gurevich1998thermodynamic}. This
restriction, however, excludes the simple random walk for \emph{all} infinite
regular graphs.

A more recent approach takes a spectral route. For a finite graph, the MERW\
can be obtained as the \emph{Doob transform} of the principal adjacency
eigenfunction $F$ by setting the transition probabilities 
\[
p_{xy}=\frac{F(y)}{\lambda F(x)} 
\]%
where $\lambda $ is the eigenvalue. One can then call MERW-s of infinite graphs the Doob transforms of arbitrary positive $\lambda$-eigenfunctions (typically a large family). This is the route taken by Duboux, Gerin and Offret in their nice recent papers \cite{thibaut2024maximum,duboux2025maximal, offret2025maximal}. They restrict $\lambda$ to be the adjacency norm on $l^2(G)$: this is reasonable for amenable graphs but again excludes the simple random walk on a regular tree as a candidate for MERW. 

Motivated by Patterson-Sullivan theory and Bowen's classical theorem on
maximal entropy flow \cite%
{bowen1972periodic,patterson1976limit,sullivan1979density}, we propose a
canonical candidate for the random walk of maximal entropy, together with a
suitable notion of entropy.

\begin{definition}
\label{def:defURWn}For a locally finite graph $G$, rooted at $o$, let the 
\emph{Uniform Random Walk} (URW) be the weak limit of the uniform
distribution on walks of length $n$ starting at $o$, when this limit exists.
\end{definition}

This means that we consider the space of infinite trajectories starting at $%
o $, put the uniform probability measure on the first $n$ steps and attempt to take a
weak limit. We show that when the URW exists, it is a random walk that
is independent of the root $o$, with transition kernel $U=(u_{xy})$ given by 
\begin{equation}
u_{xy}=\lim_{n\rightarrow \infty }\frac{W_{n-1}(y)}{W_{n}(x)}
\label{eq:defURWn}
\end{equation}%
where $W_{n}(x)$ denotes the number of walks of length $n$ starting at $x$.
For a finite graph, the URW equals the MERW and for a regular infinite
graph, URW is the simple random walk.

We need to extend the definition of URW to weighted graphs, as well, to be
able to observe phase transitions. For an arbitrary non-negative matrix $%
A=(A_{xy}\mid x,y\in V)$, one can adapt the trajectory weak limit language,
by assigning the product of edge weights to a walk, or more directly, define
the URW on $V$ as 
\[
u_{xy}=\lim_{n\rightarrow \infty }\frac{A_{xy}W_{n-1}(y)}{W_{n}(x)} 
\]%
where $W_{n}(x)=\left\langle \mathbf{1}_{x},A^{n}\mathbf{1}_{V}\right\rangle $. It
makes sense to assume that the rows of $A$ have bounded $l^{1}$ norm.


We will analyze the existence and behavior of the URW for \emph{%
loop perturbed regular graphs}.

\begin{definition}\label{def:lprg}
Let $G$ be a regular graph, $\omega \subseteq V(G)$ and $\sigma \geq 0$. We
obtain the weighted graph $G+\sigma V_{\omega }$ from $G$ by adding
self-loops at $\omega $ with weight $\sigma $.
\end{definition}

That is, we add $\sigma $ to the diagonal of the adjacency operator at the
vertex set $\omega $. Studying these graph models have a vast history,
including the theory of random Schr{\"o}dinger operators \cite{anderson1958absence,pastur1992spectra}. 

We start with the simplest case of $\omega $ being just one vertex. In operator language, this is a rank one perturbation of the adjacency matrix.
The spectral side of this perturbed graph is quite well understood \cite{simon1993spectral}. Here
we obtain a complete description of URW in terms of phase transitions and
critical behavior. Let the \emph{walk growth} of a graph be \[
\rho _{\sigma}=\lim_{n\rightarrow \infty }\sqrt[n]{W_{n}(x)}\text{.} 
\]%
The critical value $\sigma^{\ast }=\sigma^{\ast }(G, o)\in \lbrack 0,d)$ is defined in
terms of a Green function $F_o$. We suppress the details here, see Theorem~\ref{thm:Rank1} for the full result.

\begin{mtheorem} \label{thm:rankone} 
Let $G$ be a $d$-regular graph. For every $\sigma \geq 0$ and $o\in V(G)$,
the URW on $G+\sigma V_{\{o\}}$ exists.

\begin{itemize}
\item For $\sigma<\sigma^{\ast }$, $\rho _{\sigma}=d$ and the URW is transient.

\item For $\sigma>\sigma^{\ast }$, the URW is positive recurrent and is localized
near $o$, with a stationary measure that decays exponentially with the
distance from $x$. In this phase, $\rho _{\sigma}$ is a strictly monotonely
increasing function of $\sigma$.

\item For $\sigma=\sigma^{\ast }$, the URW is positive recurrent if the Green function 
$F_o$ is $\ell ^{2}$ and null recurrent otherwise.
\end{itemize}
\end{mtheorem}

For the Euclidean lattices $G=\mathbb{Z}^{d}$, this yields the following.
For $d=1,2$, we have $\sigma^{\ast }=0$, so localization happens immediately. For 
$d\geq 3$ there is a nontrivial transient phase, that is, $\sigma^{\ast }>0$. At
the critical value $\sigma=\sigma^{\ast }$, the URW is null recurrent for $d=3,4$ and
is positive recurrent for $d\geq 5$. Note that for general graphs, the well-studied emergence of an $l^2$ bound state  does \emph{not} coincide with any of the URW phases above. 

For a finite graph, Theorem \ref{thm:rankone} seems to be vacuously true: we have $\sigma^{\ast
}=0$, that is, only the positive recurrent phase survives and the URW
equals the unique MERW. The picture changes, though, if we take a \emph{%
sequence} of finite $d$-regular graphs $(G_{n})$, with $\omega _{n}\subseteq
V(G_{n})$ and $\sigma _{n}\geq 0$ and consider the Benjamini-Schramm (BS)
limit of $\mathrm{MERW}(G_{n}+\sigma _{n}V_{\omega _{n}})$. 

We first analyze the \emph{delocalized} phase, where the noise amplitude is small. Our next 
theorem makes the case that from the graph limit perspective, URW is a
reasonable candidate for MERW for infinite graphs.

\begin{mtheorem} \label{thm:mambrane}
Let $(G_{n})$ be a sequence of finite $d$-regular graphs with spectral gap and let $\omega _{n}\subseteq V(G_{n})$. Assume $%
G_{n}+V_{\omega _{n}}$ Benjamini-Schramm converges to the random rooted, looped graph $(G+V_{\omega },o)$. Then for all 
\[
\sigma < \lim_{n\rightarrow \infty }\inf (d-\lambda _{2}(G_{n})) 
\]%
the URW of the discrete graphs $G+\sigma V_{\omega }$ exist a.s. and we have 
\[
\mathrm{MERW}(G_{n}+\sigma V_{\omega _{n}})\rightarrow \mathrm{URW}(G+\sigma V_{\omega }) 
\]%
in Benjamini-Schramm convergence. 
\end{mtheorem}

The idea behind Theorem \ref{thm:mambrane} is to use a 'membrane argument' on the limiting graphing $\lG+\sigma \lV_{\omega}$, showing that applying the leafwise random noise $\sigma \lV_{\omega }$ to the principal eigenfunction of $L^{2}(\lG)$ (the constant $1$ function) will not push the function out of
place enough to stop the global adjacency operator having an isolated atom in its spectrum. This
implies (see Theorem \ref{thm:loopproc}) that the graphing adjacency operator of $\lG+\sigma \lV_{\omega }$ still has spectral gap and a unique principal eigenfunction $F$ in $L^2(\lG+\sigma \lV_{\omega })$. This then yields that the $\mathrm{URW}$ on the leaves of the limiting graphing exists, is the Doob transform of $F$ and by spectral gap, equals the limit of MERW-s of the finite graphs. Note that we need some condition on $\sigma$ in Theorem \ref{thm:mambrane}, as in Theorem \ref{thm:staple} we show that for large $\sigma$ the result will not hold. 

As we will see in Theorems \ref{thm:graphing} and \ref{thm:loopproc}, $\mathrm{URW}(G+\sigma V_{\omega })$ in Theorem \ref{thm:mambrane} also has
maximal entropy among \emph{all} the random walks on the graphing $\lG+\sigma \lV_{\omega}$ and it is a unique maximizer in a strong sense. In order to state these properly, we first need an entropy notion that works in all phases. We will use a local (rooted) version of KS entropy, that for finite, irreducible Markov chains, gives back the original notion.

\begin{definition}
Let $G$ be a locally finite weighted graph with adjacency matrix $(A_{xy})$, let 
$P=(p_{xy})$ be a random walk on $G$ and let $o\in V(G)$. For a walk $%
w=(x_{0},\ldots ,x_{n})$ starting at $x_0=o$ let  
\[
a(w)=\prod\limits_{i=1}^{n}A_{x_{i-1}x_{i}}\text{ and }%
p(w)=\prod\limits_{i=1}^{n}p_{x_{i-1}x_{i}}\text{.}
\]%
Let the \emph{%
walk entropy of length }$n$\emph{\ at }$o$ be 
\[
H_{P}^{n}(o)=-\sum\limits_{w}p(w) \log \frac{p(w)}{a(w)}
\]%
and let the \emph{walk entropy rate at }$o$ be   
\[
h_{P}(o)=\lim_{n\rightarrow \infty }\frac{1}{n}H_{P}^{n}(o)
\]%
when the limit exists. 
\end{definition}

For unweighted graphs, $h_{P}(o)$ is the Shannon entropy rate of the $P$-random walk starting at $o$. Using the above weighted notion was suggested in Duda's PhD Thesis \cite[Eq.,(3.16)]{Duda2012}, with the difference that he considered the expected one step entropy production wrt a stationary measure -- something we may not have in the localized phase. 

The natural upper bound for $h_{P}(o)$, playing the role of topological
entropy, is the log of the walk growth 
\[
h_{\mathrm{top}}(o)=\lim_{n\rightarrow \infty }\frac{1}{n}\log W_{n}(o)
\]%
assuming the growth exists. For connected weighted $G$, the existence and value of $%
h_{\mathrm{top}}=h_{\mathrm{top}}(o)$ is independent of $o$. 

\begin{proposition}\label{prp:entropies} 
$h_{P}(x)\leq h_{\mathrm{top}}$ for all $x$. If $h_{P}(x)$ exists for all $x$, it is $P$-harmonic in $x$. 
$h_{P}(x)$ may vary with $x$, but having maximal entropy $h_{P}(x)=h_{\mathrm{top}}$ is independent of $x$. For the $P$-random walk trajectory $(X_n)$, $h_{P}(X_n)$ converges a.s.  
\end{proposition}

For graphings, we can say much more.

\begin{mtheorem} \label{thm:graphing}
Let $
\lG=(X,\mu,A )$ be an ergodic, bounded degree weighted graphing and let $\mathrm{A}_{\lG}:L^{2}(X,\mu )\rightarrow L^{2}(X,\mu )$ be its adjacency operator. Let $(G,o)$ be the random rooted graph coming from $\lG$. Then for almost all $o \in (X,\mu )$ the topological entropy $h_{\mathrm{top}}((G,o))$ exists and equals $\log\left\Vert \mathrm{A}_{\lG}\right\Vert=\log \rho_\lG$. 

If the spectrum of the Graphings adjacency verifies $\Sigma_{L^2(X)}(A_{\lG})\setminus \{\rho_G\}  \subset [-\lambda,\lambda]$ for some $\lambda<\rho_G$, then for almost all $o \in (X,\mu)$ the URW of $(G,o)$ exists a.s., and has maximal entropy walk rate 
\[
h_{U}(o)=h_{\mathrm{top}}
\]%
Moreover, if $P$ is a random walk on the graphing $\lG$ of entropy walk rate $h_{\mathrm{top}}$, with a finite stationary measure on $\nu$, then $P=\mathrm{URW}(\lG)$. 
\end{mtheorem} 
Note that by a finite stationary measure on a graphing, we mean a measure that is absolutely continuous to $\mu$. Otherwise, from the point of view of the graphing, it is not visible. Theorem \ref{thm:graphing} directly applies to connected finite graphs and gives back the original theorem on MERW. The existence of $h_{\mathrm{top}}$ follows from a recent result of the first author, Fraczyk and Hayes on walk growth of unimodular random rooted graphs \cite{abert2024co}. 

Note that the unique maximizer part of Theorem \ref{thm:graphing} only holds for random walks on the full \emph{graphing} and not for random walks on the \emph{leaves}. Indeed one can perturb e.g. the SRW on the $d$-regular tree at a set of vertices of density $0$ without changing its (maximal) walk entropy rate $\log(d)$. As we will see later, environmental limits suggest a way back to uniqueness. 

We will now connect Theorem \ref{thm:mambrane} and Theorem \ref{thm:graphing} via the following.  

\begin{mtheorem} \label{thm:loopproc}
Let $G$ be a unimodular vertex transitive $d$-regular graph and $\omega $ an $\mathrm{Aut}(G)$-invariant random subset of $V(G)$, with global spectral radius $\rho _{g}$. Then for all $\sigma < d-\rho _{g}$, the adjacency operator on the graphing $\lG+\sigma \lV$ has spectral gap.

The random leaves $(G+\sigma V_{\omega },o)$ admit a URW a.s. that is the unique entropy maximizer among random walks on the graphing $\lG+\sigma \lV$ with a finite stationary measure. 
\end{mtheorem}

Theorem \ref{thm:loopproc} is trivial when $G$ is amenable or when the invariant process $\omega $ does
not have spectral gap. When $\omega $ is a (factor of) i.i.d. process on $G$%
, we have $d-\rho _{g}=d(1-\rho )$ where $\rho $ is the spectral radius of
simple random walk on $G$, so our result applies for a wide range of $\sigma$. 

It is easy to produce examples of countable graphs and random walks on it where the topological entropy, the walk entropy rate or the URW do not exist. Even when assuming they all do, it is not clear whether the URW always has maximal entropy $h_{\mathrm{top}}$. What we can show is that the URW, if exists, is a special type of walk that is uniform on bridges.  

\begin{definition}
Let $G$ be a connected, unweighted graph. A \emph{Doob walk on }$G$ is a
random walk $P=(p_{xy})$ on $G$ such that for all $x,y\in V(G)$ and $k>0$,
and walks $w_{1},w_{2}$ of length $k$ from $x$ to $y$, we have $%
p(w_{1})=p(w_{2})$. The \emph{energy} of a Doob walk is $\rho = 1/\sqrt{p_{xy}p_{yx}%
}$ where $(x,y)\in E(G)$. 
When $G$ is weighted, with adjacency matrix $A_{xy}$, we extend the definitions as follows: we have 
\[
\frac{p(w_{1})}{a(w_1)}=\frac{p(w_{2})}{a(w_2)} 
\]
for walks $w_1,w_2$ as above and the energy is defined as 
\[
\rho = \sqrt{A_{xy}A_{yx}/p_{xy}p_{yx}}
\]
Let $\mathrm{Doob}(G,\rho )$ denote the space of
Doob walks on $G$ with energy $\rho$.
\end{definition} 

It is easy to check that the energy of a Doob walk does not depend on the
edge. Note that a Doob walk of energy $\rho $ on a \emph{tree} is just a
walk with $p_{xy}p_{yx}=1/\rho ^{2}$ for all edges $(x,y)$. The space $%
\mathrm{Doob}(G,\rho )$ is compact under pointwise convergence. In \cite{thibaut2024maximum} it is
shown that fixing a root $o$, Doob walks of energy $\rho $ are exactly the
Doob transforms of positive adjacency $\rho $-eigenfunctions, normalized to
have value $1$ at $o$ (see also Proposition ()).

\begin{proposition} \label{propdoob}
Let $G$ be a connected, weighted, bounded degree graph. If the URW of $G$ exists, it is a
Doob walk of energy $\exp (h_{\mathrm{top}}(G))$. 
\end{proposition} 

For a finite graph, this again gives the old result that the Kolmogorov-Sinai
entropy of the MERW\ equals the log of the adjacency norm. 

Doob walks of energy $d$ give a large family of random walks on $T_d$ of maximal walk entropy rate $\log(d)$. In fact, one can get a full interval of walk entropy rates by varying the hitting measure. We argue that this shows that in the non-amenable setting it would not be a natural idea to call all these walks MERW-s, as \cite{thibaut2024maximum} suggests in their setting. 

\smallbreak


\noindent \textbf{Remark.} It makes sense to ask why we give a new name (Doob walk) to something that is just the Doob transform (or h-transform) of a positive eigenfunction. While on individual discrete graphs this is true, on \emph{graphings} it fails in general, as there are Doob walks on graphings that can not be integrated out to eigenfunctions on the graphing. A core suggestion of this paper is that considering limits of Doob walks instead of limits of eigenfunctions gives a better insight. 

\smallbreak 

Until now, with the exception of Theorem \ref{thm:rankone}, our results dealt with the \emph{delocalized phase}, that is, when a finite stationary measure exists. Anderson localization (see e.g. \cite{andersonsur})
trained us to expect that for large enough noise,
strong localization should happen and the URW (or at least it having finite stationary measure) should stop to exist.

\begin{question}
Does there exist $\sigma >0$ and $0<p<1$ such that $T_{d}+\sigma V_{\omega }$
($\omega $ is $(p,1-p)$ i.i.d) does not have URW, a.s.? If it exists, does it have a finite stationary measure? 
\end{question}

In its current form, the loop model and its URW at $\sigma = \infty $ do not have a straightforward definition, but looking at it as a killed process gives one a solid infinite $\sigma $ version, with hard i.i.d. obstacles. This
leads to the following.

\begin{question}
Let $T$ be a bounded degree, infinite Galton-Watson tree. Does $T$ have URW,
a.s.?
\end{question}

Note that it is not true that every ergodic bounded degree graphing has URW. E.g. $\mathbb{Z}+V_{\omega }$ ($\omega$ is i.i.d.) is a counterexample, see the end of the introduction.


We first demonstrate the localized phase on a natural example: the \emph{canopy tree} $CT_d$. A convenient visualization is to fix a boundary point $\xi \in \partial T_d$ of the regular tree and consider the subgraph induced by the horoball centered at $\xi$:
\[
CT_d = \{ x \in T_d : b_{\xi}(x) < 0 \},
\]
where $b_{\xi}$ is the Busemann function (see Section~\ref{sec:canopy}). For $x \in V(CT_d)$, let $|x|$ denote the minimal distance from $x$ to a leaf. The group $\mathrm{Aut}(CT_d)$ acts transitively on each level set $\{x : |x| = k\}$. Assigning the root to level $|x| = k$ with probability proportional to $(d-1)^{-k}$ makes $(CT_d, x)$ a unimodular random rooted graph, which is the Benjamini--Schramm limit of balls in $T_d$.

Before stating the result, we need to define the \emph{environmental limit}, that is, the environment seen by the particle. See \cite{env1, env2, env3} for previous literature. 

\begin{definition}\label{def:envlimit}
Let $(G,P,o)$ be a bounded degree rooted graph with a random walk. Let $P^n(o,\dd x)$ be the endpoint measure of the $P$-random walk starting at $o$
and let 
\[
\nu _{n}^{P}(o, \dd x)=\frac{1}{n}\sum_{i=1}^{n}P^i(o,\dd x)\text{.}
\]%
The environmental limit of $G$ wrt $P$ from $o$ is  
\[
\lim_{n\rightarrow \infty }(G,P,x_n) = (G^\prime,P^\prime,o^\prime)
\]%
where the root $x_n$ has law $\nu_{n}^P(o,\dd x)$. 
\end{definition}

That is, we root $(G,P)$ at a $\nu_{n}^P(o,\dd x)$-random root and take the weak limit of rooted graphs + random walks. The limiting random rooted graph is not necessarily unimodular, but it is $P^\prime$-stationary. It turns out that the environmental limit controls the existence of a finite stationary measure and when no such measure exists, serves as a substitute. 

\begin{proposition}\label{prop:envlimit}
Let $G=(X,\mu, A)$ be a bounded degree graphing with a random walk $P$, such that the environmental limit $(G^{\prime },P^{\prime },x)$ of $P$ starting at $x \in (X,\mu)$ exists a.s. and is independent of $x$. Then the walk entropy rate $h_{P}=h_{P}(x)
$ is independent of $x$ and equals $E_{(G^{\prime },P^{\prime },x))}H_{P^{\prime
}}^{1}(x)$, the expected one-step entropy production of $P^{\prime }$ in $G^{\prime }$.
Moreover, if $X$ admits a $P$-stationary probability measure $\lambda$ that is absolutely continuous to $\mu$, then $(G^{\prime },P^{\prime },x)$ equals $(G,P,o)$ where $o \in (X,\lambda)$, that is, the law of $x$ is $\lambda$. 
\end{proposition}


It turns out that while simple random walk is recurrent on $%
CT_{d}$, the URW is \emph{strongly transient}, that is, for the URW-random
trajectory $(X_{n})$, we have  $\lvert X_{n}\rvert  \rightarrow
\infty $, a.s. 

\begin{mtheorem} \label{thm:canopy}
The URW exists for the canopy tree $CT_{d}$ and is strongly
transient. Projected on the levels, this random walk is given by 
$$ u_{l,l \pm 1}= \frac{l \pm 1}{2l},\qquad l= 1,2,\ldots$$
We have 
\[
h_{U}=h_{\mathrm{top}}(CT_{d})=\log (2\sqrt{d-1})
\]%
that is, the $\textmd{URW}$\ has maximal walk entropy rate. The environmental limit of $CT_d$
with respect to the $\textmd{URW}$ equals the unique Doob walk on $T_{d}$ of energy $2\sqrt{d-1}$ which picks a fixed boundary point, and goes towards it with probability $1/2$. In particular, $CT_d$ does not admit an absolutely continuous finite stationary measure wrt its URW.
\end{mtheorem}

\noindent \textbf{Remark.} As we discussed, uniqueness of the entropy maximizer never holds for deterministic nonamenable graphs, as one can always perturb an entropy maximizing walk at a negligible set of vertices.  
The same trick works on a graphing -- like the canopy tree -- where an entropy maximizer is strongly localized: one can make small changes to the random walk along the environment sequence, to be never seen again by the particle, hence eventually not affecting the walk entropy rate. A potential surviving question here is whether the \emph{environmental limit} of entropy maximizers is unique for (suitable) graphings. 

\bigbreak 

Now let us run the following experiment. Let $A$ and $B$ be two large,
connected $d$-regular graphs. The MERW\ of their disjoint union equals SRW.
Let us enforce localization on the MERW\ of $B$ by perturbing $B$, say, at a
single site, to have adjacency norm $\rho > d$. This will, of course, not affect the MERW on $A$. Now connect $A$ and $B$ by a bridge, that is, erase an edge in both and cross over. The effect on the adjacency norm is negligible, but it will strongly change the MERW on the whole $A$. Indeed, the new MERW will be a Doob walk of energy $\sim\rho$ which makes $p_{xy}p_{yx} \sim 1/\rho^2$ for \emph{all} edges $xy$. This effect is size and distance independent, and as the sizes go to infinity, it becomes an external field. The following theorem
implies that the adjacency norm completely controls the nature of this
change and that no symmetry breaking will happen at the critical value $d$.

\begin{mtheorem} \label{thm:staple} 
Let $G_{n}$ be a sequence of finite regular simple graph with Benjamini-Schramm (BS) limit $(G,o)$. Assume that be the adjacency norm $\rho_n$ of the loop perturbed graph $G_{n}+\sigma_n V_{\omega_n}$ converge. \newline 
1) If $\lim_{n\rightarrow \infty }\rho_n=d$ then 
\[
\mathrm{MERW}(G_{n}+\sigma_n V_{\omega_n}) \xrightarrow[n]{BS} (\mathrm{SRW}(G),o)
\]%
2) If $\lim_{n\rightarrow \infty }\rho_n=\rho >d$, then no subsequence of $\mathrm{MERW}(G_{n}+\sigma_n V_{\omega_n})$ converges to $(\mathrm{SRW}(G),o)$ in Benjamini-Schramm.
\end{mtheorem}

A particularly interesting case is when 2) holds together with  
\[
\left\vert \omega_n \right\vert
=o(\left\vert V(G_{n})\right\vert )\text{.}     
\]
Here, the Kolmogorov-Sinai entropy of $\mathrm{MERW}(G_{n}+\sigma_n V_{\omega_n})$ converges to $\log (\rho)>\log (d)$. On the other hand, any (subsequential) limit of $\mathrm{MERW}(G_{n})$ will be a Doob walk $P$ of energy $\rho $ on a $d$-regular graphing $G$ with walk entropy rate $h_{P}<\log (d)$. On the finite level, the surplus entropy is stored at the higher degree vertices as the stationary measure concentrates on them. But where does the surplus entropy go in the limit? The random walk $P$ does remember $\rho $ but as energy, not as leafwise entropy. We will discuss this in a forthcoming paper, also giving explicit estimates for  Theorem \ref{thm:staple}. 

Ergodic theory gives us some extra information in this phase: while the MERW stays away from SRW and flows towards the high degree node(s), its local drift will \emph{not} be dominated by a single direction. If that would be the case, the lift of the limiting walk to $T_d$ would hit the boundary at a Dirac measure, producing an invariant probability measure on it -- violating a classical theorem. 

In the case of random $d$-regular graphs, we expect a strong concentration result:

\begin{question}
Let $(G_{n})$ be a random $d$-regular graph on $n$ vertices, let $v_{n}$ be a uniform random vertex of $G_{n}$ and fix $\sigma >d-1-(d-1)^{-1}$. Is there a graphing $\lG_{\rho }$ with a random walk $P_{\rho}$ of energy $\rho = \rho(\sigma)$ on $G_{\rho }$ such that 
\[
\mathrm{MERW}(G_{n}+\sigma V_{\{v_{n}\}})\rightarrow P_{\rho }\text{ } 
\]%
in Benjamini-Schramm convergence, a.s.? Can one directly define $P_{\rho }$
/ its hitting measure?
\end{question}

Another related question is which invariant random Doob walk(s) of energy $\rho > d$ on $T_d$ have maximal walk entropy rate? 

The most natural random graph model where localization of MERW is a built-in feature is the giant component of the Erdos-Renyi graph $G(n,c/n)$ where $c>1$ is a constant. One can attempt to take the BS limit of the MERWs, which, if exists, will live on a Poisson Galton-Watson tree. The issue here is that the maximal degree will tend to infinity with $n$ and so is the adjacency norm $\rho$. Since the MERW is a Doob walk of energy $\rho$, any weak limit will have infinite energy, causing $p_{xy}*p_{yx} = 0$ for all edges $(x,y)$. That is, every edge will choose an orientation. 

\begin{question}
For $c>1$ let $G_n$ be the giant component of the Erdos-Renyi graph $G(n,c/n)$. Does $\mathrm{MERW}(G_n)$ Benjamini-Schramm converge a.s.? 
\end{question}

\noindent \textbf{Remark.} The current paper focuses on nonamenable graphs, but it is quite natural to ask whether the weighted graphs 
\[
\mathbb{Z}^{d}+\sigma V_{\omega }\text{ \ }(\sigma >0\text{, }\omega \text{
is }\{p,1-p\}\text{ i.i.d. for }0<p<1)
\]
admit a Uniform Random Walk. We can show that for $d=1$ the answer is no, for
any $\sigma > 0$ and $0<p<1$. For higher dimensions, we have only partial results, and at
this point, the existence of URW is not clear. This direction is fundamentally different from the
scope of the current paper. Not just in its toolset and the
underlying difficulties: opposed to the nonamenable setting, for $\mathbb{Z}^{d}+\sigma V_{\omega }$ it is meaningful
to ask about the existence of URW at different \emph{scales}, using a sequence of embeddings of $%
\mathbb{Z}^{d}$ into $\mathbb{R}^{d}$. We will publish our results in a forthcoming paper. 

\bigskip   

\noindent \textbf{Acknowledgements.} The authors thank Francois Ledrappier and Tianyi Zheng for helpful discussions. They thank Tom Hutchcroft and Omer Tamuz for suggesting the Mass Transport Scheme used in Theorem \ref{thm:staple}. The first author thanks Vilas Winstein with whom he made initial computations on the canopy tree. 

\bigskip

The paper is organized as follows. Section~2 contains preliminary remarks about the URW. Section~3 is devoted to the analysis of the URW under a single-site perturbation. Section~4 analyzes the URW on the canopy tree. Section~5 develops the membrane method in the general setting of invariant loop processes and establishes the corresponding maximal–entropy characterization on graphings. Finally, Section~6 discusses the continuity of the URW along Benjamini–Schramm convergent sequences and derives several consequences for the MERW on finite regular graphs.

\section{Preliminaries}

This section contains definitions and lemmas used throughout the paper. 

Every graph $G$ will be assumed connected and locally finite. We allow loops and edge weights. The adjacency matrix of $G$ is a symmetric non-negative array
$A_G = (A_{xy})_{x,y\in V(G)}$ with $A_{xy} > 0$ if and only if $x$ and $y$ are neighbors, denoted by $x\sim y$. We assume that both the combinatorial and weighted degree
\[
\deg(x) := \#\{y : A_{xy} > 0\}, 
\qquad 
\deg_A(x) := \sum_{y} A_{xy},
\]
are uniformly bounded: 
\[
\sup_{x \in V(G)} \max\{\deg(x),\, \deg_A(x)\} \;\le D \;<\;\infty .
\]
The adjacency operator $A_G$ acts on functions $\varphi : V(G)\to\mathbb{C}$ by
\[
(A_G \varphi)(x) = \sum_{y\sim x} A_{xy}\varphi(y).
\]
We write $\ell^2(G)=\ell^2(V(G))$ for the Hilbert space of square-summable functions with scalar product
$$
\langle \varphi, \psi\rangle := \sum_{x\in V(G)} \overline{\varphi(x)}\,\psi(x).
$$
Since $A_G$ is bounded and symmetric, it is self-adjoint. Its spectrum is denoted by $\Sigma_{\ell^2(G)}(A_G)\subset\mathbb{R}$, 
and its spectral radius is
$$
r(G) := \max \Sigma_{\ell^{2}(G)}(A_G)
       = \lim_{n\to\infty} \sqrt[\,2n\,]{\langle \mathbf{1}_x, A_G^{2n}\mathbf{1}_x\rangle}
       \qquad\text{for any } x\in V(G),
$$

A \emph{walk} is a sequence $w = x_0 x_1 \cdots x_n$ such that $x_{i-1} \sim x_i$ for all $1 \leq i \leq n$, where $n$ is called the length of $w$, denoted by $|w|$. 
We interpret $A_{xy}$ as the number of possible links that a walker located at site $x$ can choose in order to move to site $y$, with a natural extension to non-integer weights.
The \emph{total weight} of a walk $w= x_0 x_1 \cdots x_n$ is
$A(w) := A_{x_{o} x_1}\cdots A_{x_{n-1} x_n}$.

Define $W_n(x)$ to be the total weight of all walks of length $n$ starting at $x$:
\[
W_n(x) := \langle \mathbf{1}_x, A_G^n \mathbf{1}_{V(G)} \rangle = \sum_{|w|=n} A(w).
\]
A key quantity is the exponential growth rate:
\begin{equation}
\rho(G) := \lim_{n \to \infty} \sqrt[n]{ W_n(x) }.
\end{equation}
This limit, when it exists, is independent of the choice of $x \in V(G)$. Since $r$ and $\rho$ correspond to the exponential growth of return and general walks respectively, we always have
\[
\|A_G\| = r(G) \leq \rho(G) \leq d_A(G):=\sup_{x} \sum_y A_{xy}.
\]

The exponent of $\rho(G)$ is called the \emph{topological entropy} of $G$, denoted by $h_{\rm top}(G):= \log \rho(G)$.

\begin{proof}[Proof of Proposition~\ref{prp:entropies}]
Recall that for a Markov chain $P$ on $G$, the walk entropy rate $h_{P}(x)$ at $x \in G$ is the rate for the walk entropy $H_P^n(x)$ of length $n$. 
By using Jensen's inequality, for any $P$, we have 
\[
H_P^n(x) = -\sum_{w} \log p(w) \log \frac{p(w)}{a(w)}\le
\log W_n(x).
\]
Thus, whenever $h_P(x)$ and $h_{\rm top}$ exist, we have $h_P(x) \le h_{\rm top}$.
From the equation $H_P^n(x) = (PH_P^{n-1})(x) + H_P^1(x) = \sum_{k=0}^{n-1} (P^kH_P^1)(x)$, we verify that $h_P$ is a $P$-harmonic function when it is defined for all $x$.
Due to the maximum principle, once it attains the value $h_{\rm top}$, it is constant function: $h_P \equiv h_{\rm top}$. 
\end{proof}

\medskip

\begin{proposition}[Proposition~\ref{propdoob}]
    If the URW on $G$ exists, then it is a Doob walk of energy $\exp(h_{\rm top)}$. In particular, it is a Markov chain.
\end{proposition}
\begin{proof}
From the existence of the limit \eqref{eq:defURWn}, for any path $w = x_0x_1\cdots x_k$, the probability of $w$ for URW is 
\begin{align*}
    u(w)
    =\prod_{j=1}^k\lim_{n \to \infty} A_{x_{j-1}x_j}\frac{W_{n-j}(x_j)}{W_{n-j+1}(x_{j-1})}
    = a(w) \lim_{n\to\infty}\frac{W_{n-k}(x_k)}{W_{n}(x_{0})}.
\end{align*}
Therefore, the URW is a Doob walk: for any path $w_1, w_2$ of length $k$, we have 
\[
\frac{u(w_1)}{a(w_1)} = \frac{u(w_2)}{a(w_2)} = \lim_{n\to\ \infty} \frac{W_{n-k}(x_k)}{W_n(x_0)}
\]
Especially by taking $w$ with $k=2$, we verify that URW is a Doob walk of energy $\rho(G) = \exp(h_{\rm top})$.
\end{proof}

We shall study more general objects called graphings.
\begin{definition}
A (bounded, weighted) \emph{graphing} is a triple $\mathcal{G} = (X,\mu,A)$, where 
$(X,\mu)$ is a standard Borel probability space and 
$A : X\times X \to [0,\infty)$ is a symmetric Borel function satisfying the following two conditions. It has essential maximal weighted degree
\[
\operatorname*{ess\,sup}_{x\in X}
   \max \left\{ \#\{y : y\sim x\},\; \sum_{y} A_{xy} \right\} < \infty,
\]
and satisfies the \emph{Mass Transport Principle} (MTP) 
\begin{equation}\label{eq:MTP}
\int_S \deg_{\mathcal{G}}(x, T)\, \mu(\mathrm{d}x)
   = \int_T \deg_{\mathcal{G}}(x, S)\, \mu(\mathrm{d}x),    
\end{equation}
for all $S,T \in \mathcal{B}(X)$, where $\deg_{\mathcal{G}}(x,S) := \sum_{y\in S} A_{xy}.$
\end{definition}
When the weights are of the form $A_{xy}=1_{\{(x,y)\in E\}}$ for some subset $E\subset X\times X$, one recovers the usual definition of unweighted (i.e. simple) graphing \cite{Lovasz2012large}. As before, we write $x\sim y$ for $A_{xy}>0$, this defines a Borel graph on $X$. A graphing $\mathcal{G}$ is said to be \emph{ergodic} if any Borel set $S \in \mathcal{B}(X)$ that is invariant under the graphing adjacency relation (that is, $x \in S$ and $y\sim x$ implies $y \in S$) satisfies $\mu(S)\in\{0,1\}$. 

The graphing adjacency $A_\lG$ is the operator acting on $L^2(\lG) := L^2(X,\mu)$ via 
$$(A_{\lG} F)(x) = \sum_{y\sim x} A_{xy}F(y).$$
The MTP is precisely the self adjointness condition for $A_{\lG}$

Two rooted connected graphs are isomorphic, denoted by $(G,o) \cong (G^\prime, o^\prime)$, if there is a graph automorphism $\phi: G \to G^\prime$ which preserves the root $\phi(o) = o^\prime$. We denote by $\mathbb{G}_{\bullet}^D$ the set of (isomorphism class of) rooted graph $(G, o)$ with maximal degree $\le D$. The topology of Gromov-Hausdorff convergence turn $\mathbb{G}_{\bullet}^D$ into a Polish space. The same discussion applies to the class of weighted graphs.

Every (possibly weighted) graphing $\lG = (X,\mu,A)$ with bounded degree comes with maps that sends $o \in X$ to $(G,o) \in \mathbb{G}_{\bullet}^D$, obtained by restriction to the connected component of $C(o) \subset X$. We call $(G,o)$ the \emph{leaf} associated to the point $o \in X$. We obtain an adjacency operator $A_{(G,o)}$ that act on the Hilbert space $\ell^2(G,o)$. When the point $x$ follows the distribution $\mu(\dd x)$, this induce a distribution on $\mathbb{G}_{\bullet}^D$, and the MTP implies the unimodularity condition (e.g. see \cite{aldous2007processes}):
\[
\int_{\mathbb{G}_{\bullet}^D} \sum_{x\in V(G)} F(G,o,x)\, \mu(\dd(G,o))
\;=\;
\int_{\mathbb{G}_{\bullet}^D} \sum_{x\in V(G)} F(G,x,o)\, \mu(\dd(G,o)).
\]
for non negative measurable function $F$ on the space $\mathbb{G}_{\bullet\bullet}^D$ of isomorphism class of doubly rooted graphs. \\
\\
For a bounded weighted graphing $\lG$, we define $W_n(x)$ to be the total weighted of the walk of length $n$ as the root of the leaf $(G,o)$, and similarly $\rho((G, o))$. For a bounded degree graphing $\lG$, we define the topological entropy, denoted by $h_{\rm top}=h_{\rm top}(\lG,x)$, in the same way.

A random walk on a graphing is a Borel function $P \colon X\times X \to [0,1]$ with the following condition. If $A_{xy}>0$ then $p_{xy}>0$ and $\sum_{y\sim x}p_{xy} \equiv 1$ for almost all $x \in (X,\mu)$. This induces a Markov kernel on each leaf that we write $(G,P,o)$. We write $\mathbb{G}_{\bullet}^{D,M}$ the space obtained from $\mathbb{G}_{\bullet}^{D}$ by a Markov chain decoration.

We shall introduce a specific example of a graphing induced from a locally finite connected $d$-regular vertex-transitive graph $G$.
Let $\Omega=\Omega_G = \{0,1\}^V$ be the space of $\{0,1\}$-configurations on the vertex set of $G$.
For each configuration $\omega$ in $\Omega$, we define the graph decorated with a loop of weight $\sigma\ge0$ at each vertex $x$ where $\omega_{x} =1$, denoted by $G+\sigma V_{\omega}$.
In terms of the adjacency operator, we write
\[
A_{G+\sigma V_{\omega}} = A_G + \sigma \sum_{x\in\{\omega_{x}=1\}}\mathbf{1}_{x}\mathbf{1}_{x}^{*}.
\]
We will discuss in Section~\ref{sec:rankoneperturbation} on the case of graphs with a single loop.

We can realize the space of configurations as a Borel graph by identifying each configuration $\omega$ with the rooted graph $(G+\sigma V_{\omega}, o)$ with weighted loops for some fixed root $o \in G$. Two configurations $\omega, \omega^\prime$ in $\Omega$ are adjacent if and only if there is an automorphism $\phi: G \to G$ which sends $o$ to an adjacent vertex and $\phi.\omega = \omega^\prime$.

\section{Rank-One Perturbation for regular graphs}\label{sec:rankoneperturbation}

In this section, we state and prove the full theorem behind Theorem \ref{thm:rankone}. 

We assume that $G$ is an infinite, $d$-regular graph with the adjacency matrix $A=A_{G}$. Since the number of walks of length $n$ starting from any vertex $x$ is $W_n(x) = d^n$, the URW on $G$ coincides with the simple random walk, with transition probabilities $q_{xy} = \frac{1}{d} A_{xy}$.

To illustrate the rich behavior of the URW under local perturbations, we consider a one-parameter family of graphs obtained by adding a loop of weight $\sigma \in \mathbb{R}_{\geq 0}$ at a distinguished root vertex $o \in G$. We denote the resulting weighted graph by $G + \sigma V_{\{o\}}$, with adjacency operator
\[
H = A + \sigma \mathbf{1}_o \langle \mathbf{1}_o, \cdot \rangle = A + \sigma \mathbf{1}_o \mathbf{1}_o^*.
\]
Let $W_n^\sigma(x) := \langle \mathbf{1}_x, H^n \mathbf{1}_{V} \rangle$ denote the total weight of walks of length $n$ starting from $x$ in $G + \sigma V_{\{o\}}$. In this section, we derive the URW explicitly in terms of $\sigma$ and the generating function on $G$:
\[
f_{xy}(t) := \left(1 - tA\right)^{-1}_{xy}.
\] 
We note that special value $t = \frac{1}{d}$ plays a key role, as it corresponds to the Green function of the SRW on $G$. We restrict to rank-one perturbations for simplicity and for the sake of explicit computations using generalized Green functions. Nonetheless, the same qualitative behavior extends to any one-parameter family $H = A + \sigma V$, provided that $V$ has finitely supported, non-negative entries in the vertex basis of $G$.

\begin{theorem}\label{thm:Rank1}
Let $G$ be a $d$-regular graph. For every $\sigma \in \mathbb{R}_{\geq 0}$, the URW on $G + \sigma V_{\{o\}}$ exists. Define the threshold
\[
\sigma^*(G,o) := \frac{d}{f_{oo}(\tfrac{1}{d})} \in [0,d),
\]
and 
\[
\rho_\sigma := \limsup_{n \to \infty} \sqrt[n]{W_n^\sigma(x)}, \qquad
u^\sigma_{xy} := \lim_{n \to \infty} \frac{H_{xy} W_{n-1}^\sigma(y)}{W_n^\sigma(x)}.
\]
\begin{itemize}
    \item If $0 \leq \sigma < \sigma^*(G,o)$, then $\rho_\sigma = d$ and the URW is transient, with transition probabilities
    \[
    u^\sigma_{xy} = \frac{H_{xy}}{d} \cdot \frac{1 - \frac{\sigma}{d} f_{oo}(\tfrac{1}{d}) + \frac{\sigma}{d} f_{oy}(\tfrac{1}{d})}{1 - \frac{\sigma}{d} f_{oo}(\tfrac{1}{d}) + \frac{\sigma}{d} f_{ox}(\tfrac{1}{d})}.
    \]
    \item If $\sigma \geq \sigma^*(G,o)$, then the growth rate $\rho_\sigma$ is strictly increasing in $\sigma$, given by $\rho_\sigma = \left[ t f_{oo}(t) \right]^{-1} \left( \frac{1}{\sigma} \right) \geq d$. The URW is recurrent, with transition probabilities
    \[
    u^\sigma_{xy} = \frac{H_{xy}}{\rho_\sigma} \cdot \frac{f_{oy}(\tfrac{1}{\rho_\sigma})}{f_{ox}(\tfrac{1}{\rho_\sigma})}.
    \]
\end{itemize}

For $\sigma > \sigma^*(G,o)$, the process is localized near the root, with a stationary distribution that decays exponentially with the distance from $o$. At the critical value $\sigma = \sigma^*(G,o)$, the URW positive recurrent if the Green function $F_o(x)= f_{ox}(\tfrac{1}{d})$ is square summable and null recurrent otherwise. 
\end{theorem}

In quantum mechanics, one often considers the Schrödinger operator $dI - A - \sigma\,\mathbf{1}_o\mathbf{1}_o^*$, and it is folklore that an $\ell^2$ eigenvector emerges for sufficiently large $\sigma$.
We argue that Theorem~\ref{thm:Rank1} describes a random-walk analogue of this localization phenomenon.
While the mathematical tools underlying the result are standard, the Markov chain arising in the transient phase appears to be new, even on Euclidean lattices. A surprising phenomenon arises in the non-amenable setting: there exists an intermediate regime in which a top $\ell^2$ eigenvector does exist, yet the URW is \emph{not} its Doob transform\footnote{In this phase, the spectral radius $\rho_{\mathrm{spec}}$ is strictly smaller than the walk growth $\rho_{\mathrm{rw}}$, a key distinction from previous attempts \cite{thibaut2024maximum,duboux2025maximal} to extend MERW to the infinite-volume setting.}. See Figure~\ref{fig:simple-axis} below for the resulting phase diagram.

\paragraph{Example 1: The Euclidean case of integer lattices $\mathbb{Z}^n$.}
For $n = 1,2$, we have $ f_{oo}(\tfrac{1}{d}) = \infty $, so $ \sigma^*(\mathbb{Z}^n) = 0 $. This implies that adding any positive loop weight at $ o $ immediately localizes the URW. For $ n \geq 3 $, there is a non-trivial transient phase for $ 0 \leq \sigma < \sigma^*(\mathbb{Z}^n) $. At the critical value $ \sigma = \sigma^*(\mathbb{Z}^n) $, the stationary measure is proportional to the square of the Green function:
\[
\pi(x) \propto f_{ox}(\tfrac{1}{d})^2.
\]
It is well known that the green function admits the asymptotic
\[
f_{ox}(\tfrac{1}{d}) = c_n |x|_2^{2-n} + O(|x|_2^{-n}),
\]
where $ |x|_2 $ denotes Euclidean distance and $ c_n $ is a dimension-dependent constant~\cite[Theorem 4.3.1]{lawler2010random}. This shows that the URW at criticality is null recurrent for $ n = 3,4 $ and positive recurrent for $ n \geq 5 $.

\paragraph{Example 2: The hyperbolic case of the regular trees $T_d$.}
Due to the recursive structure of trees, explicit formulas are available. At the root of a tree $\mathcal{T}$, every closed walk decomposes uniquely into sequence of walks in the pending sub trees $\mathcal{T} \setminus \{x\}$. We denote the adjacency of the forest obtained by removing the site $x$ by $A|_{\mathcal{T} \setminus \{x\}}$. One has
\[
(1 - tA)^{-1}_{xx} = \frac{1}{1 - \sum_{y \sim x} tA_{xy} \cdot (1 - tA_{\mathcal{T} \setminus \{x\}})^{-1}_{yy} \cdot tA_{yx}}.
\]
More generally, for any walk from $x$ to $y$ along the geodesic $x = x_0, x_1, \dots, x_l = y$, we obtain
\[
(1 - tA)^{-1}_{xy} = (1 - tA)^{-1}_{xx} \cdot \prod_{i=1}^l \left[ tA_{x_{i-1}x_i} \cdot (1 - tA|_{\mathcal{T} \setminus \{x_{i-1}\}})^{-1}_{x_ix_i} \right].
\]

Let $\mathcal{T} = T_d$ be the infinite $d$-regular tree, and set $q := d - 1$. 
Removing any vertex leaves $d$ copies of the $q$-ary tree $\mathcal{T}_q$. Together with the above fomulae we get
\begin{align*}
  &f_{xx}(t) = \frac{1}{1 - dt^2 h(t)}, \\
  &h(t) = \frac{1}{1 - qt^2 h(t)} = \frac{1 - \sqrt{1 - 4qt^2}}{2qt^2}, \\
  &f_{xy}(t) = f_{xx}(t) \cdot (t h(t))^{d(x,y)}.
\end{align*}

Let $H = A + \sigma \mathbf{1}_{o}\mathbf{1}_{o}^*$ be the adjacency operator with a loop of weight $\sigma$ at the root $o$ and $(X_n)_{n}$ the associated URW. By spherical symmetry, the behavior is characterized by the local drift toward the root
\[
\mathbb{E}[d(o, X_{n+1}) - d(o, X_n) \mid X_n = x],
\]
as a function of $d(o, x)$. Using $h(\tfrac{1}{d}) = \tfrac{d}{d-1}$, the critical threshold above which $\rho_\sigma \neq d$
\[
\sigma^*(T_d) = \frac{d}{f_{oo}(\tfrac{1}{d})} = (d - 1) - \frac{1}{d - 1}.
\]
For $\sigma > \sigma^*(T_d)$, since $f_{ox}(\tfrac{1}{\rho_\sigma}) \propto \left(\frac{1}{\rho_\sigma}h(\tfrac{1}{\rho_\sigma}) \right)^{d(o,x)}$, the URW has a constant negative drift toward the root. For $\sigma < \sigma^*(T_d)$, the local drift approaches $\tfrac{d - 2}{d}$ as $x$ moves away from the root, matching the linear speed of the SRW.

We may also compute the threshold $\sigma_{\ell^2}$ (see Lemma~\ref{lmm:Rank1l2}) above which $r(H) \neq r(A)$:
\[
\sigma_{\ell^2}(T_d) = \frac{2\sqrt{d - 1}}{f_{oo}(\tfrac{1}{2\sqrt{d - 1}})} = \sqrt{d - 1} - \frac{1}{\sqrt{d - 1}}.
\]
This illustrates that for non-amenable graphs, such as $T_d$, there exists an intermediate phase $ \sigma_{\ell^2} < \sigma^* $ where the URW is \emph{not} the Doob transform of the top $ \ell^2 $-eigenvector of $H$, as the associated eigenvalue $r(H)$ remains strictly below $d$.

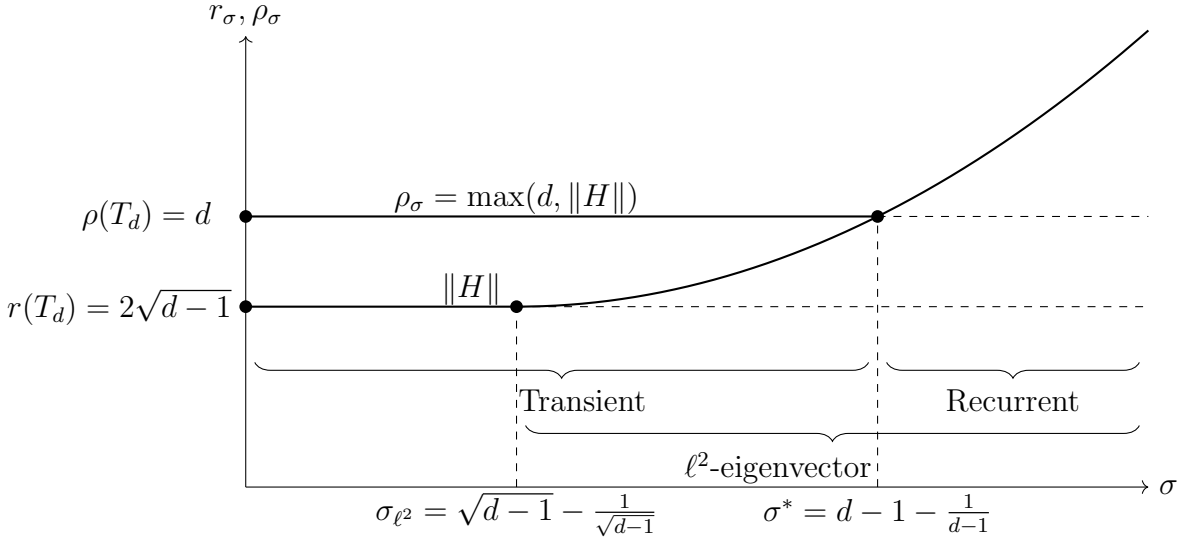
\begin{figure}[H]
    \hspace*{-3cm}
  \begin{tikzpicture}[scale=1.2]

\draw[->] (0,0) -- (10,0) node[right] {$\sigma$};
\draw[->] (0,0) -- (0,5) node[above] {$r_\sigma,\rho_\sigma$};

\draw[dashed] (7,3) -- (10,3) node[above left] {};
\draw[dashed] (3,2) -- (10,2) node[below right] {};

\draw[thick, domain=0:7, smooth, variable=\x] 
    plot ({\x}, {3 }) node[right] {};
\draw[thick, domain=0:3, smooth, variable=\x] 
    plot ({\x}, {2 }) node[right] {};
\draw[thick, domain=3:10, smooth, variable=\x] 
    plot ({\x}, {2 +0.0625*(\x-3)*(\x-3)}) node[right] {};

\fill (0,3) circle (2pt) ; 
\fill (0,2) circle (2pt); 
\fill (3,2) circle (2pt);
\fill (7,3) circle (2pt); 

\draw[dashed] (3,0) -- (3,2);
\draw[dashed] (7,0) -- (7,3);

\node at (3,-0.3) {$\sigma_{\ell^2}=\sqrt{d-1}-\tfrac{1}{\sqrt{d-1}}$};
\node at (7,-0.3) {$\sigma^*=d-1 - \frac{1}{d-1}$};
\node at (-1.1,3) {$\rho(T_d)=d$};
\node at (-1.4,2) {$r(T_d)=2\sqrt{d-1}$};
\node at (3,3.2) {$\rho_\sigma=\max(d,\| H \|)$};
\node at (2.5,2.2) {$\| H\|$};
\node at (7,3) {};

\draw[decorate,decoration={brace,amplitude=6pt,mirror}, yshift=-0.5pt]
(0.1,1.4) -- (7-0.1,1.4) node[midway,below=6pt] {$\quad$ Transient};
\draw[decorate,decoration={brace,amplitude=6pt,mirror}, yshift=-0.5pt]
(7+0.1,1.4) -- (10-0.1,1.4) node[midway,below=6pt] {Recurrent};

\draw[decorate,decoration={brace,amplitude=6pt,mirror}, yshift=-0.5pt]
(3+0.1,0.7) -- (10-0.1,0.7) node[midway,below=6pt] {$\ell^2$-eigenvector \qquad \quad\,};

\end{tikzpicture}
  \caption{\textit{Phase diagram} of the rank one perturbation for $T_d$ and its URW.}
  \label{fig:simple-axis}
\end{figure} 

\subsection{Proof of Theorem \ref{thm:Rank1}}

In case of rank one perturbation $H-A=\sigma\mathrm{1}_o\mathrm{1}_o^*$, generating function $f_{xy}^\sigma(t)=(1-tH)^{-1}_{xy}$ of walks in $G + \sigma V_{\{o\}}$ can be computed explicitly in terms of the analogous quantity in the unperturbed graph $G$. 
    
\begin{lemma}\label{lmm:Rank1comb}
The generating function $f^\sigma_{xy}(t)=(1-tH)_{xy}^{-1}$  of walk in $G + \sigma V_{\{o\}}$ satisfies
\[
f^\sigma_{xy}(t) = \sum_{n=0}^\infty (H^n)_{xy} t^n = f_{xy}(t) + \frac{f_{xo}(t) \cdot \sigma t \cdot f_{oy}(t)}{1 - \sigma t f_{oo}(t)}, \qquad x, y \in \lG
\]
In particular,
\begin{equation}\label{eq:rank1}
    f^\sigma_{ox}(t) = \frac{f_{ox}(t)}{1 - \sigma t f_{oo}(t)}.
\end{equation}
\end{lemma}

\begin{proof}
The standard proof relies on the second resolvent identity. Here we sketch an alternative proof using a simple combinatorial argument. We expand the generating function as a sum over paths $\gamma$ from $x$ to $y$
\[
f^\sigma_{xy}(t) = \sum_{\gamma} H(\gamma) t^{|\gamma|}. 
\]
We organize this sum according to the number $ k \geq 0 $ of visits to the loop at $ o $ along $ \gamma $. If $ k = 0 $, the path avoids the loop entirely. Its contribution is simply $ f_{xy}(t) = (1 - tA)^{-1}_{xy} $. Otherwise, we have $f_{xo}(t) \cdot \sigma t \cdot \left( \sigma t f_{oo}(t) \right)^{k - 1} \cdot f_{oy}(t)$ for $k \ge1$.

Summing over $ k = 1 $ gives the geometric series:
\[
f_{xy}^\sigma(t)=f_{xy}(t)+\sum_{k = 1}^\infty f_{xo}(t) \cdot \sigma t \cdot \left( \sigma t f_{oo}(t) \right)^{k - 1} \cdot f_{oy}(t) = f_{xy}(t)  + \frac{f_{xo}(t) \cdot \sigma t \cdot f_{oy}(t)}{1 - \sigma t f_{oo}(t)}.
\]
\end{proof}

We begin with some spectral results from rank-one perturbation theory. While these results are well known, we have not found them stated in this precise form in the literature, and we therefore include them here for completeness. Our proof of the spectral gap deviates slightly from the classical approach, as it does not rely directly on explicit computations of the resolvent of $H$ in the complex plane. 

\begin{lemma}[Spectral analysis of rank-one perturbation]\label{lmm:Rank1l2}
Let $r_\sigma = r(G+\sigma V_{\{o\}})$ and $r(A)=r(G)$.
\[
\sigma_{\ell^2}(G,o) := \frac{r(A)}{f_{oo}(\tfrac{1}{r(A)})} \in [0,\infty).
\]
Then $r_\sigma > r(A)$ if and only if $\sigma > \sigma_{\ell^2}(G,o)$. In this case, $r_\sigma$ is given by the inverse function $r_\sigma = \left[ t f_{oo}(t) \right]^{-1}(1/\sigma)$, and $H$ has a spectral gap $r_\sigma - r(A)$. Moreover, $H$ admits a unique top $\ell^2(G)$-eigenvector, proportional to the generalized Green function $F^\sigma_o(x):=f_{ox}(\tfrac{1}{r_\sigma})$.
\end{lemma}

\begin{proof} Note that the spectral radius is continuous, since $r_\sigma-r(A) \leq \| H-A\| = \sigma$. As $H$ is positive, irreducible, and aperiodic when $\sigma > 0$, we have $r_\sigma = \lim \sqrt[n]{(H^n)_{xy}} = \|H\|$.
From Lemma~\ref{lmm:Rank1comb}, one has
\[
f_{oo}^\sigma(t) = \sum_n (H^n)_{oo} t^n = \frac{f_{oo}(t)}{1 - \sigma t f_{oo}(t)}
\]
The first singularity on the positive real axis corresponds to $r_\sigma$. If $\sigma$ is small, this is $1/r(A)$, otherwise, there is a solution $t=1/r_\sigma<1/r(A)$ of
$$1-\sigma tf_{oo}(t)=0,$$
this happens only if $\sigma>(\frac{1}{r(A)}f_{oo}(\tfrac{1}{r(A)}))^{-1}=: \sigma_{\ell^2}$. This conclude the claim on the spectral radius.

We now assume that $\sigma>\sigma_{\ell^2}$. To prove existence of the spectral gap, let $\lambda \in \Sigma_{\ell^2(G)}(H)\setminus [-r(A),r(A)]$. Consider a Weyl sequence associated to $\lambda$. That is, a sequence $\varphi_n \in \ell^2(G)$ with $\|\varphi_n \|=1$, and $\psi_n :=(H - \lambda) \varphi_n$ is such that $\| \psi_n \| \to 0$. 
\[(H - \lambda)\varphi_n = (A - \lambda)\varphi_n + \sigma \varphi_n(o) \mathbf{1}_o = \psi_n \to 0.
\]
If $\varphi_n(o) \to 0$, then $(A - \lambda)\varphi_n \to 0$, contradicting that $\lambda \notin \Sigma_{\ell^2(G)}(A)$. So that by compactness, we may assume that $\varphi_n(o) \to c \neq 0$. As $\lambda$ is away from the spectrum of $A$, we can consider the bounded inverse of $A-\lambda$
\[\varphi_n = (\lambda - A)^{-1} (\sigma \varphi_n(o) \mathbf{1}_o - \psi_n).
\]
Taking inner products with $\mathbf{1}_o$, passing to the limit, we get
\[c = c \sigma (\lambda - A)^{-1}_{oo} \Rightarrow \sigma^{-1} = (\lambda -A )^{-1}_{oo}=\frac{1}{\lambda} f_{oo}(\tfrac{1}{\lambda}).\]
Hence, $\lambda = r_\sigma$ and $r_\sigma$ is an isolated point of the spectrum. \\

Assuming $\varphi \in \ell^2(G)$ verifying $(H-r_\sigma) \varphi=0$, it is easy to see that
\[\varphi = \varphi(o)\frac{\sigma}{r_\sigma} (1 - A/r_\sigma)^{-1} \mathbf{1}_o \Rightarrow \varphi(x) \propto f_{ox}(\tfrac{1}{r_\sigma}).\]
\end{proof}

\begin{proof}[Proof of Theorem \ref{thm:Rank1}]

In the weighted graph $G+\sigma V_{\{o\}}$, every node has (weighted) degree $d$, except the root $o$ whose degree is $d+\sigma$. Any path from $x$ has therefore $d$ or $d+\sigma$ ways to be extended by one step. This gives a recursion for the walk count
$$W_n^\sigma(x) = (d+\sigma)H_{xo}^{n-1} + d\sum_{y\not =o } H_{xy}^{n-1} = d W_{n-1}^\sigma(x) + \sigma H_{xo}^{n-1} =  d^n + \sigma \sum_{k=0}^{n-1} H_{ox}^{k}d^{n-1-k},$$ 
therefore
$$W_n^\sigma(x) = d^n\left( 1+ \frac{\sigma}{d} \sum_{k=0}^{n-1} \frac{H_{ox}^{k}}{d^k}\right).$$
If $r_\sigma\leq d$ one has
$$\sqrt[n]{W_n^\sigma(x)}=d\left(1+\frac{\sigma}{d}O(n)\right)^{1/n} \to d,$$
otherwise $r_\sigma>d$, and by the previous lemma $H_{ox}^{k} = cr_\sigma^k + O(d^k)$ for some constant $c=c_x \neq 0$ depending on the top eigenvector. 
$$\sqrt[n]{W_n^\sigma(x)}=d\left(1+\frac{\sigma}{d}\left(\frac{r_\sigma}{d}\right)^{n-1}+O(n)\right)^{1/n} \to r_\sigma.$$
We obtain the following formula for exponential growth of walk 
$$\rho_\sigma = \lim_{n} \sqrt[n]{W_n^\sigma(x)} = \max\{d, r_\sigma \}.$$
Note that the transition occurs at
$$r_\sigma>d \iff \sigma > \sigma^*(G,o) := \frac{d}{f_{oo}(\tfrac{1}{d})} \geq \frac{r(A)}{f_{oo}(\tfrac{1}{r(A)})} = \sigma_{\ell^2}(G,o).$$
This conclude the discussion for the walks growth $\rho_\sigma$. We need to show that the ratio limit $W_{n-1}(y)/W_n(x)$ exists.\\

There are several cases.

\textbf{Case 1 ($\sigma > \sigma^*(G,o)$):} Then $r_\sigma = \rho_\sigma > d$. Using $\sigma^*(G,o) \geq \sigma_{\ell^2}(G,o)$, Lemma \ref{lmm:Rank1l2} and $d \geq r(A)$ we obtain that $H_{ox}^k \sim \varphi^\sigma(o) \varphi^\sigma(x) \rho_\sigma^k + O(d^k)$, where we introduced the normalized eigenvector $\varphi^\sigma \propto f_{ox}(\tfrac{1}{\rho_\sigma})$. Therefore,
\[W_n^\sigma(x) \sim \varphi^\sigma(o) \varphi^\sigma(x) \cdot d^{n-1} \cdot \frac{\sigma}{d} \cdot \frac{(\rho_\sigma/d)^n - 1}{\rho_\sigma/d - 1} \Rightarrow \frac{W_{n-1}^\sigma(y)}{W_n^\sigma(x)} \to \frac{\varphi^\sigma(y)}{\rho_\sigma \varphi^\sigma(x)}.
\]

\textbf{Case 2 ($0 < \sigma < \sigma^*(G,o)$):} Since $ \sigma^*(G,o)>0$ , the SRW on $G$ is transient and $f_{ox}(\tfrac{1}{d})$ converges. Since $\sigma<\sigma^*(G,o)$ the numerator in $\eqref{eq:rank1}$ is not singular and $f_{ox}^\sigma(\tfrac{1}{d})$ converges. From Lemma~\ref{lmm:Rank1comb}:
\[\frac{W_{n-1}^\sigma(y)}{W_n^\sigma(x)} \to \frac{1}{d}\frac{ 1+ \frac{\sigma}{d} \sum_{k=0}^{\infty} \frac{H_{oy}^{k}}{d^k} }{ 1+ \frac{\sigma}{d} \sum_{k=0}^{\infty} \frac{H_{ox}^{k}}{d^k} }=  \frac{1}{d} \cdot \frac{1 - \tfrac{\sigma}{d} f_{oo}(\tfrac{1}{d}) + \tfrac{\sigma}{d} f_{oy}(\tfrac{1}{d})}{1 - \tfrac{\sigma}{d} f_{oo}(\tfrac{1}{d}) + \tfrac{\sigma}{d} f_{ox}(\tfrac{1}{d})}.
\]

\textbf{Case 3 ($\sigma = \sigma^*(G, o)>0$):} the critical case  is more involved. Here $f_{ox}(1/d)$ is finite but $f^\sigma_{ox}(1/d) = \infty$ because the numerator in $\eqref{eq:rank1}$ is precisely singular at $t=\tfrac{1}{d}$. To analyze this, we adapt the classical renewal theory of Markov process to the matrix $H$, by decomposing paths in terms of the first visit to $o$.

By time reversal of path, the coefficient of the symmetric matrix $(H^n)_{ox}$ equals the sum over all paths $\gamma$ from $x$ to $o$ of length $n$:
\[
(H^n)_{ox} = \sum_{\gamma : x \to o, |\gamma| = n} H(\gamma).
\]
Let $k \leq n$ be the first time $\gamma$ visits $o$. Each such path decomposes uniquely as $\gamma = \tilde{\gamma}_1 \gamma_2$, where $\tilde{\gamma}_1$ is a path in $G$ from $x$ to $o$, and $\gamma_2$ is a closed path starting and ending at $o$ in $G+\sigma V_{\{o\}}$.

Before reaching the vertex $o$ the matrix coefficient of $H$ and $A$ along the path coincide. We might replace $H$ by $A$ along $\tilde{\gamma}_1$ and get:
\[
\frac{(H^n)_{ox}}{d^n} = \sum_{k=1}^n \left( \sum_{\tilde{\gamma}_1} \frac{A(\tilde{\gamma}_1)}{d^k} \sum_{\gamma_2} \frac{H(\gamma_2)}{d^{n-k}} \right).
\]
The first sum is the probability that SRW from $x$ hits $o$ for the first time at time $k$, denoted $q_{xo}(k)$. The second term is $H^{n-k}_{oo}/d^{n-k}$, denoted $h(n-k)$. Thus,
\[
\frac{(H^n)_{ox}}{d^n} = \sum_{k + l = n} q_{xo}(k) h(l).
\]
Since $\sum_{k } q_{xo}(k) \leq 1$ and $\sum_{l } h(l) = f^\sigma_{oo}(1/d) = \infty$. We now compute
\[
\sum_{k=0}^{n-1} \frac{H^k_{ox}}{d^k} = \sum_{k=0}^{n-1} \sum_{l=0}^k q_{xo}(l) h(k - l) = \sum_{l=0}^{n-1} q_{xo}(l) \sum_{m=0}^{n-1-l} h(m).
\]
Letting $g(l,n) := \sum_{m = 0}^{n-1-l} h(m) / \sum_{m = 0}^{n-1} h(m)$, we have $g(l,n) \to 1$ as $n \to \infty$ and
\[
\frac{1}{\sum_{k=0}^{n-1} h(k)} \sum_{k=0}^{n-1} \frac{H^k_{ox}}{d^k} = \sum_{l=0}^{n-1} q_{xo}(l) g(l,n) \to \sum_{l=0}^\infty q_{xo}(l),
\]
by the monotone convergence theorem. 
Since $W_n^\sigma(x) \sim \frac{\sigma}{d} \sum_{k = 0}^{n-1} \frac{H^k_{ox}}{d^k}$ we get
\[
\frac{W_{n-1}^\sigma(y)}{W_n^\sigma(x)} \sim \frac{1}{d} \cdot \frac{\sum_{k=0}^{n-1} \frac{H_{oy}^k}{d^k}}{\sum_{k=0}^{n-1} \frac{H_{ox}^k}{d^k}} \to \frac{1}{d} \cdot \frac{\sum_{k=0}^\infty q_{yo}(k)}{\sum_{k=0}^\infty q_{xo}(k)} = \frac{1}{d} \cdot \frac{f_{oy}(1/d)}{f_{ox}(1/d)},
\]
using $f_{ox}(1/d) = f_{oo}(1/d) \sum_{k=0}^\infty q_{xo}(k)$ by the strong Markov property of SRW from $x$ on $G$, at the first hitting time of $o$ .

This concludes the proof that the URW exists for all $\sigma \geq 0$.
The transition probabilities take the form
\[
u_{xy}^\sigma = \frac{H_{xy}}{\rho_\sigma} \cdot \frac{\varphi^\sigma(y)}{\varphi^\sigma(x)}.
\]
The coefficients $u^\sigma_{xy}$ depend continuously on $\sigma$ at the transition $\sigma = \sigma^*(G,o)$. To determine whether the URW is recurrent or transient, we examine the Green function of the Markov operator $P^\sigma$:
\[
(1 - P^\sigma)^{-1}_{xx} = \varphi^\sigma(x)^{-1}  \cdot \sum_{n=0}^\infty \rho_\sigma^{-n} (H^n)_{xx}  \cdot \varphi^\sigma(x
) = f^\sigma_{xx}(1/\rho_\sigma).
\]
In particular,
\[
(1 - P^\sigma)^{-1}_{oo} = \frac{f_{oo}(1/\rho_\sigma)}{1 - \tfrac{\sigma}{\rho_\sigma} f_{oo}(1/\rho_\sigma)}.
\]
The URW is transient when this expression is finite, i.e., when $\sigma < \sigma^*(G,o)$. For $\sigma \geq \sigma^*(G,o)$, the stationary distribution is
\[
\pi(x) \propto \left( f_{ox}(1/\rho_\sigma) \right)^2 = \left( (1 - A/\rho_\sigma)^{-1}_{ox} \right)^2.
\]
When $\sigma > \sigma^*(G,o)$, we have $\rho_\sigma > d$ and exponential decay in the distance $d(o,x)$:
\[
f_{ox}(1/\rho_\sigma) = \sum_{n = d(o,x)}^\infty (d/\rho_\sigma)^n = \left( \frac{d}{\rho_\sigma} \right)^{d(o,x)} \cdot \frac{1}{1 - d/\rho_\sigma}.
\]
Finally, at the critical point $\sigma = \sigma^*(G, o)$, $\rho_{\sigma^*(G,o)}=d$, the URW is positive recurrent if and only if $f_{ox}(\tfrac{1}{d}) \in \ell^2(G)$.
\end{proof}

\section{Environmental limit and the canopy tree.}\label{sec:canopy}

In this section we discuss the canopy tree and prove Theorem \ref{thm:canopy}. 

On a graphing (and therefore on any unimodular random rooted graph), it was shown in \cite{abert2024co} that the walk growth exists almost surely. Since the definition of the random walk operator $P$ on a graphing does not depend on the choice of the underlying measure, one cannot expect an equally general result for the entropy rate associated with $P$.

However, whenever the environment seen from the random walk stabilizes, it induces a stationary ergodic process for the increments of the walk, and standard arguments imply that the entropy per step converges. This was claimed in Proposition \ref{prop:envlimit}, and we now turn to its proof.

\begin{proof}[Proof of Proposition \ref{prop:envlimit}] 
Given $o\in (X,\mu)$, the random walk on the graphing produces a walk on the leaf, denoted $X_n^P(o)$. Since
\[
X_{n+1}^P(o)=X_1^P\bigl(X_n^P(o)\bigr),
\]
the environmental limit satisfies the stationarity equation in distribution
\[
(G',P',x)\;=_d\;(G',P',X_1^{P'}(x)).
\]
Because degrees are uniformly bounded, $H_P^1(x)$ is bounded and therefore integrable. Hence, by weak convergence of $(G,P,x_n)$ to $(G',P',x)$ and bounded convergence of the one-step entropy functional,
\[
\lim_{n\to\infty}\frac{1}{n}H_P^n(o)
= E_{(G',P',o')\sim\mu'}\,H_{P'}^{1}(o')
= -\int_{\mathbb G_{\bullet}^{D,M}}\sum_{y\sim x} p_{xy}\log p_{xy}\,\mu'(\mathrm d(G,P,x)).
\]

Now assume that the walk on the graphing admits a $P$-stationary probability measure $\lambda\ll\mu$. For $o'\in(X,\lambda)$, stationarity gives
\[
(G,P,X_n^P(o'))=_d(G,P,o') \quad\text{for all }n.
\]
Since the environmental limit exists and is independent of the starting point, we also have
\[
(G,P,X_n^P(o'))\;\Rightarrow\;(G',P',x).
\]
Therefore $(G',P',x)=_d(G,P,o')$.
\end{proof}

With this proposition at hand, we turn to the analysis of the canopy tree. This example makes the underlying mechanism especially transparent, as the limiting environment can be described explicitly. Fix a geodesic ray $o x_1 x_2 x_3 \cdots$ converging to a boundary point $\xi \in \partial T_d$ of the regular tree. The associated Busemann function is
\[
b_\xi(y)=\lim_{n\to\infty}\bigl(d(y,x_n)-n\bigr).
\]
The \emph{canopy tree} is the induced subgraph given by the horoball
\[
CT_d=\{x\in T_d:\; b_\xi(x)<0\}.
\]
Note that $CT_d$ contains the points $x_i$ and they converge to the unique boundary point $\xi$. Every non-leaf vertex of $CT_d$ has degree $d$. For $x\in CT_d$, we denote by 
\[
|x|\in\{1,2,\ldots\}
\]
the minimal distance from $x$ to a leaf. Note that this is simply $-|b_\xi(x)|$, and that one has $|x_k|=k$ along the geodesic towards $\xi$. The automorphism group $\mathrm{Aut}(CT_d)$ acts transitively on each level $\{x:\ |x|=k\}$. From a rooted graph perspective, one can project the graph on the set of levels
\[
\mathbb N=\{1,2,\ldots\},
\]
on which the adjacency operator $A_{CT_d}$ becomes the tridiagonal operator $B$ defined by
\[
B_{k,k+1}=1,\qquad B_{k+1,k}=d-1.
\]
Although $B$ is not symmetric, it becomes self-adjoint on the weighted space
$\ell^2(\mathbb N,p)$ with
\[
p_k=\frac{d-2}{(d-1)^k}.
\]
This distribution is the limiting distribution of the distance from a leaf in balls of increasing radius in $T_d$.

\begin{proof}[Proof of Theorem~\ref{thm:canopy}]
Let $D=\operatorname{Diag}\bigl((\sqrt{d-1})^k\bigr)$ and let $J$ be the standard Jacobi
operator on $\mathbb N$ with $J_{k,\ell}=1_{\{|k-\ell|=1\}}$. A direct
calculation shows
\[
B=\sqrt{d-1}\, D\, J\, D^{-1}.
\]
It is well known that the sine transform, which identifies $\mathbf{1}_{k}$ with 
$\sqrt{2/\pi}\,\sin(k\theta)$, diagonalizes $J$, which becomes the multiplication operator
$2\cos\theta$ on $L^2(0,\pi)$; see e.g.\ \cite{simon2005orthogonal}. This allows an integral representation of the number of walks. Let $x\in CT_d$ be at level $|x|=k$. Then
\begin{align*}
W_n(x)
&=\sum_{y\in V(CT_d)} (A_{CT_d}^n)_{xy}
=\sum_{\ell=1}^\infty (\sqrt{d-1})^{\,n+k-\ell} J^n_{k,\ell} \\
&=\frac{2^{\,n+1}(\sqrt{d-1})^{\,n+k}}{\pi}
   \int_0^\pi (\cos\theta)^n\sin(k\theta)\,
   \sum_{\ell=1}^\infty \frac{\sin(\ell\theta)}{(\sqrt{d-1})^\ell}\,d\theta.
\end{align*}
Set
\[
I_n(k)=\int_0^\pi e^{n g(\theta)} h_k(\theta)\,d\theta,
\qquad g(\theta)=\log|\cos\theta|,
\]
and
\[
h_k(\theta)
=\operatorname{sign}(\cos(n\theta))\,
  \sin(k\theta)\,
  \frac{\sqrt{d-1}\,\sin\theta}{d-1-2\sqrt{d-1}\cos\theta+1}.
\]
Then
\[
\frac{W_{n-1}(x)}{W_n(x)} = \frac{1}{2}\,\frac{I_{n-1}(k)}{I_n(k)}.
\]
Only neighborhoods of $0$ and $\pi$ contribute to the large-$n$ asymptotics of $I_n$. We have
\[
h_k(\theta) = \frac{k\,\theta^2}{(\sqrt{d-1}-1)^2 }+ O(\theta^3),
\qquad \theta\to 0,
\]
and
\[
h_k(\theta)
= \frac{(-1)^{k+n} k\,(\theta-\pi)^2}{(\sqrt{d-1}+1)^2 }
  + O((\theta-\pi)^3),
\qquad \theta\to\pi.
\]
A classical Laplace analysis gives
\[
I_n(k)\sim \frac{\sqrt{\pi}}{\sqrt2\,n^{3/2}}
\left(
\frac{\sqrt{d-1}\,k}{(\sqrt{d-1}-1)^2}
+\frac{\sqrt{d-1}\,k\,(-1)^{\,n+k}}{(\sqrt{d-1}+1)^2}
\right).
\]

Thus, if $y\sim x$ is at level $k+1$, we obtain
\[
\lim_{n\to\infty}\frac{W_{n-1}(y)}{W_n(x)}
= \frac{k+1}{2k} = u_{xy}.
\]
Similarly, if $k\ge2$ and $y$ is one of the $d-1$ neighbors at level $k-1$, we obtain
\[
\lim_{n\to\infty}\frac{W_{n-1}(y)}{W_n(x)}
=\frac{1}{d-1}\,\frac{k-1}{2k} = u_{xy}.
\]
This shows the existence of the URW. It is the Doob transform of simple random walk with respect to
\[
\rho=2\sqrt{d-1}, \qquad F(x)=F(|x|)=|x|(\sqrt{d-1})^{|x|}.
\]
Since $F\notin \ell^2(\mathbb N,p_k)$, the stationary measure
$\pi_k=F(k)^2 p_k$ is infinite, so the URW admits no finite stationary
distribution on $CT_d$.  Projecting the walk to the level sets produces the
birth--death chain
\[
u_{k,k\pm1}=\frac{k\pm1}{2k},\qquad k\ge1,
\]
which is transient. In fact, it goes to infinity at speed $\sqrt{n}$. This implies $|X_n|\to\infty$ almost surely for every
initial condition $X_0 \in CT_d$.  To describe the environment seen from the walk, we reroot
$(CT_d,\mathrm{URW})$ at $X_n$.  As $|X_n|\to\infty$, the rooted graph
$(CT_d,\textmd{URW},X_n)$ converges locally. The environmental limit is $(T_d,P^\xi,o)$, where $P^\xi$ moves toward $\xi$ with probability $1/2$ and otherwise steps to one of
the $d-1$ remaining neighbors with probability $1\big/2(d-1)$. The entropy rate of the URW equals the
expected one-step entropy on the environmental limit
\[
h_{\mathrm{URW}}
=H^1_{P^\xi}(o)
=\log(2\sqrt{d-1})
=h_{\mathrm{top}}(CT_d),
\]
The theorem holds. 
\end{proof}

\section{URW on Graphings \label{sec:URWgraphingloopproc}} 

In this section we prove Theorem \ref{thm:graphing} and Theorem \ref{thm:loopproc}. 

When one adapts Perron-Frobenius theory to infinite graphs, the main obstacle is that the constant function $\mathbf{1}_{V(G)}$ is no longer an element of the Hilbert space $\ell^{2}(G)$. This prevents the use of spectral tools to analyze the convergence of the ratio limits of the walk growth
$$
W_n(x)=\langle \mathbf{1}_{x}, A^n \mathbf{1}_{V(G)} \rangle.
$$
Remarkably, for a graphing $\lG=(X,\mu,A)$, the situation is much simpler. The constant function $\mathbf{1}_X$ lies in the domain of the adjacency $\mathrm{A}_\lG$, a self adjoint operator on $L^{2}(\lG)$. In the presence of a spectral gap, we shows that the power iteration converges and the URW exists just as in the finite case.

\begin{proof}[Proof of Theorem \ref{thm:graphing}]

We begin with the topological entropy claim. Consider graphing with bounded weight. In \cite[Theorem 5.5]{abert2024co}, it is shown using a sub-multiplicative argument that the limit $\lim_{n}(W_{n}(o))^{1/n}$ exists for almost all $o\in(X,\mu)$. Although the proof is presented for simple graphings, the very same argument applies to the weighted case without modification. This limit is invariant under moving the root, and by ergodicity it is an almost sure constant.

The fact that this constant coincides with the norm of the adjacency operator follows from two observations. Since $A_G$ preserves positivity, one has the Gelfand type formula
$$
\lVert \mathrm{A}_{\lG}\rVert=\lim_{n}\lVert \mathrm{A}_{\lG}^n\mathbf{1}_X\rVert^{1/n},
$$
where $\mathbf{1}_X$ is the constant function in $L^{2}(\lG)$. Furthermore,
$$
(\mathrm{A}_{\lG}^n\mathbf{1}_X)(o)=W_n(o).
$$
Thus the topological entropy exists for almost every instance of the root, and
$$
h_{\mathrm{top}}((G,o))=\log\lVert \mathrm{A}_{\lG}\rVert.
$$
We denote $\rho_\lG=\lVert \mathrm{A}_{\lG}\rVert$.

\smallskip

By assumption, the spectrum of the graphing adjacency operator verifies 
\[
\Sigma_{L^2(X,\mu)}(A_G) \subset [-\lambda,\lambda] \cup\\
\{\rho_\lG \},
\]
for some constant $\lambda<\rho_\lG$. By positivity of $A_{\mathcal G}$ and ergodicity of $\mathcal G$, the eigenspace corresponding to $\rho_{\mathcal G}$ is one-dimensional. Let $F$ be the associated nonnegative and $L^{2}$-normalized eigenvector.
The identity $A_\lG F=\rho_\lG F$ in $L^2(\lG)$ translates leaf-wise on $\ell^2((G,o))$ to\footnote{We stress that $F$ is usually not in $\ell^2((G,o))$; it is therefore not an eigenvector of the leaf wise adjacency operator $A_{(G,o)}$ but only a generalized eigenfunction, in the sense of locally finite graphs.}
$$
\sum_{y\sim x}A_{xy}F(y)=\rho_\lG F(x),\qquad x,y\in(G,o).
$$
From this, the event $F(o)>0$ is invariant by moving the root and by ergodicity, we obtain that $F(o)$ is positive for almost every $o\in (X,\mu)$. We claim that the URW exists on $(G,o)$ and corresponds to the Doob transform with respect to $F$. The operator $\rho_\lG^{-n} A_\lG^n $ converges (in strong topology) to the spectral projection onto the eigenspace generated by $F$. This implies the $L^2(\lG)$-convergence
$$
\rho_\lG^{-n} W_n(\cdot)=\rho_\lG^{-n} A_\lG^n \,\mathbf{1}_X \;\xrightarrow[n]{L^2(\lG)}\; \langle \mathbf{1}_X, F\rangle_{L^2(\lG)}\,F(\cdot).
$$
The spectral gap hypothesis gives
\begin{equation}\label{eq:convestimate}
\left\lVert \rho_\lG^{-n} W_n - \langle \mathbf{1}_X,F\rangle_{L^2(\lG)} F \right\rVert
\leq \left(\tfrac{\lambda}{\rho_\lG}\right)^n,
\end{equation}
so by a Borel-Cantelli type argument we obtain the point-wise convergence
$$ \lim_n \rho_\lG^{-n} W_n(o) = F(o), \qquad \textmd{$o\in (X,\mu)$ a.e.}$$
This implies the existence of the ratio limit of walk growth. For $x,y\in X$, we set
$$u_{xy}=\lim_{n}\frac{A_{xy}W_{n-1}(y)}{W_n{x}}=\frac{A_{xy}F(y)}{\rho_\lG F(x)},$$
if $x\sim y$ and zero otherwise. The kernel $U=(u_{xy})$ defines a random walk on the (weighted) graphing $\lG =(X,\mu,A)$, whose leaf-wise action corresponds to the $\mathrm{URW}$ on each leaf $(G,o)$. This process is stationary with respect to the measure
$$ \pi(\mathrm{d}x) = F(x)^2\,\mu(\mathrm{d}x).$$ 
To see this, we remark that for any bounded Borel function $\varphi$,
\begin{align*}
\int_X (U \varphi)\,\mathrm{d}\pi
  &= \int_X (U\varphi)(x)\,F(x)^2\,\mu( \dd x) \\
  &= \frac{1}{\rho_\lG}\int_X F(x)\,A_\lG(F\varphi)(x)\,\mu(\dd x) \\
  &= \frac{1}{\rho_\lG}\int_X (A_\lG F)(x)\,F(x)\varphi(x)\,\mu(\dd x) \\
  &= \int_X \varphi(x)\,F(x)^2\,\mathrm{d}\mu( \dd x)
   = \int_X \varphi\,\mathrm{d}\pi.
\end{align*}
We denote $(X_n^U(o))_{n\ge0}$ the uniform random walk on the leaf $(G,o)$ starting at $X_0^U(o)=o$. The $n$-th step $X_n^U(o)$ is distributed following the measure $U^n(o,\dd x)$. By definition
$$H_{U}^n(o) = H(X_1^U(o),X_2^U(o)\ldots X_n^U(o)) = \sum_{i=0}^{n-1} \int_{X} H_{U}^1(x_k) U^{k}(o,\dd x_k).$$
If $o^\prime$ is distributed following $\pi$, we have the $L^1$ and almost sure convergence
$$ \lim_n \frac{1}{n} \sum_{i=1}^n -u_{X_{i-1}^U(o^\prime)X_i^U(o^\prime)}\log \left( \frac{u_{X_{i-1}^U(o^\prime)X_i^U(o^\prime)}}{A_{X_{i-1}^U(o^\prime)X_i^U(o^\prime)}} \right) \to -\int_{X} \sum_{y\sim x}\log \left(\frac{u_{xy}}{A_{xy}} \right)u_{xy} \pi( \dd x).$$
This is a standard application of Birkhoff’s Ergodic theorem, We refers to \cite[proposition 2.1]{benjamini2012ergodic} in case of stationary sequence of rooted graph with respect simple random walk. We get that
$$h_{U}(o^\prime)=\lim_n \tfrac{1}{n}H_{U}^n(o^\prime) = \int_X H_{U}^1(x)\pi(\dd x), \qquad \textmd{$o^\prime\in (X,\pi)$ a.e.}$$
The fact that the one step entropy production is the topological entropy follows from
$$-\int_{X} \sum_{y\sim x}\log \left(\frac{u_{xy}}{A_{xy}} \right)u_{xy} \pi( \dd x) = \log(\rho_\lG) -\int_{X} \sum_{y\sim x}\log \left(\frac{F(y)}{F(x)} \right)u_{xy} \pi( \dd x) = h_{\mathrm{top}}.$$
\medskip

Let $P$ be any random walk on the graphing $\lG=(X,\mu,A)$, such that the entropy rate exists for almost all $o \in (X,\mu)$. We assume that $P$ has a stationary measure $\nu$ absolutely continuous with respect to $\mu$. We denotes its kernel $p_{xy}$ its kernel. For each $x$, define the one step Kullback--Leibler divergence 
\begin{align*}
g(x)&=\sum_{y\sim x} -\log \left(\frac{{u_{xy} }}{p_{xy}} \right) p_{xy}\\
&=\sum_{y\sim x} \log \bigl(\frac{ p_{xy} }{ A_{xy} }\bigr)p_{xy}
-\sum_{y\sim x} \log \bigl(\frac{ u_{xy} }{ A_{xy} }\bigr)p_{xy}\\
&=-H_P^1(x)+\log\rho_\lG-\sum_{y\sim x} p_{xy}\log ( F(y)/F(x) ).
\end{align*}
By assumption, $h_P(o)=\lim_n\frac{1}{n}{H}_P^n(o)$ exists for almost all $o \in (X,\mu)$, and it should coincide with the one step entropy production averaged over the stationnay measure $\nu$, therefore
\begin{align*}
   \int_X g(x) \,\nu(\mathrm{d}x) &= \int_X\sum_{y\sim x} (\log\rho_\lG - \log  \left( \frac{F(y)}{F(x)} \right)+ \log \left(\frac{p_{xy}}{A_{x,y}}\right))p_{xy}\,\nu(\mathrm{d}x)\\
   &=\log \rho_\lG -  \int_X h_P(o)\nu(\dd x) \geq 0. 
\end{align*}
Since $h_P(o)=h_\mathrm{top}$, we get that $g(x)=0$ implying $(p_{xy})_{y\sim x}=(u_{xy})_{y\sim x}$ for almost all $x\in(X,\nu)$. But the support of $\nu$ is a rerooting–invariant Borel subset of $X$ of positive measure. By ergodicity, this set must have full measure $\mu$, and therefore $P = U$.
\end{proof}

We now turn to the proof the Membrane method for invariant loop process.

\begin{proof}[Proof of Theorem \ref{thm:loopproc}]
We first realize the invariant loop process as a graphing. 
Let $G$ be a unimodular vertex-transitive $d$-regular graph and $\omega\subseteq V(G)$ an $\mathrm{Aut}(G)$-invariant random subset. Consider the space
\[
X = \{0,1\}^{V(G)},
\]
endowed with the product topology, and let $\mu$ be the $\mathrm{Aut}(G)$-invariant probability measure induced by the law of~$\omega$.

The graphing $\mathcal{G}=(X,\mu,A)$ is defined by putting an edge between $\omega \sim \omega^\prime$, (that is $(\omega,\omega^\prime)\in E$), when two configurations are obtained by moving the root $o$ of $G$ along an edge. When $G$ is not a Cayley graph, one may need to break symmetries using an i.i.d.\ labeling (Bernoulli graphing). This construction is standard and preserves unimodularity.

Consider the graphing adjacency operator $A_\lG = (\mathbf{1}_{\{(x,y)\in E\}})$ of the loop process. Since $G$ is $d$-regular and the graphing is ergodic, the constant function $\mathbf{1}_X$ is the unique top eigenfunction
\[
A_{\mathcal{G}} \mathbf{1}_X = d\,\mathbf{1}_X.
\]
By assumption, this operator has spectral gap $\delta := d - \rho_g$, where the global spectral radius satisfies $\rho_g := \bigl\| A_{\mathcal{G}}\big|_{\mathbf{1}_X^\perp} \bigr\| < d$. Therefore
\[
\Sigma_{L^{2}(X,\mu)}(A_{\mathcal{G}})
\subset [-\rho_g,\rho_g] \cup \{d\}.
\]
We consider the diagonal graphing $\lV=(V,X,\mu)$ where $V=\{(\omega,\omega) \colon o \in \omega \}$. Note that since $V$ is diagonal, the MTP is trivially verified. Given $\sigma>0$, we consider the weighted graphing $\lH_\sigma = \lG+\sigma \lV=(X,\mu,H_\sigma)$ in a natural way. Note that for each $\omega\in X$, the leaf $(H_\sigma,\omega)=(G+\sigma V,\omega)$ corresponds to the loop perturbed regular graph $G+\sigma V_\omega$ as in definition \ref{def:lprg}. We claim that the graphing adjacency operators 
$$A_{\lH_\sigma}=A_{\mathcal{G}}+\sigma A_{\mathcal{V}}$$
has a top eigenvector with a spectral gap.
\medskip

When $\sigma$ is small, this is fairly standard operator theory, we include here for convenience. Let $z$ in the complex plane with $|z-d|=(d-\rho_g)/2$, and $\sigma<(d-\rho_g)/2$, on has $\| \sigma A_{\lV}  (A_{\mathcal{G}} - z)^{-1}\| < 1 $, hence the Born series
\[
(A_{\lH_\sigma} - z)^{-1}
= (A_{\mathcal{G}} - z)^{-1}
  \sum_{k=0}^\infty \bigl(- \sigma A_{\mathcal{V}} (A_{\mathcal{G}} - z)^{-1}\bigr)^k
\]
converge in norm. This in particular implies that the spectrum of $A_{\lH_\sigma}$ is in the Minkowski sum between the spectrum of $A_\lG$ and $[\pm \sigma]$. In particular 
\begin{equation}\label{eq:Mink}
  \Sigma_{L^2(X,\mu)}(A_{\lH_\sigma}) \subset [-\rho_g-\sigma,\rho_g+\sigma] \bigsqcup [d-\sigma,d+\sigma]   
\end{equation}
Define the Riesz projection by the integral on $\gamma=\{z : |z-d|=(d-\rho_g)/2\}$ given by
\[
P_\sigma := \frac{1}{2\pi i} \int_\gamma (A_{\lH_\sigma} - z)^{-1}\,\mathrm{d}z.
\]
The integral should be understood as Riemann sum limit, converging in the strong topology of operator. It depends analytically on $\sigma$ as long as $\gamma$ stays in the resolvent set. By norm-continuity of $P_\sigma$, its rank is constant, see \cite{reed1978iv}. At $\sigma=0$, $P_0$ is the orthogonal projection onto $\mathbb{C}\mathbf{1}_X$, so $\mathrm{rank}\,P_0=1$. There is therefore a unique top eigenvector $F_\sigma$ associated to $A_{\lH_\sigma}$, detached from the rest of the spectrum.
\medskip

So far, we have only shown spectral gap up to $\sigma<(d-\rho_g)/2$ as, above this value the two interval overlap in the right side of \eqref{eq:Mink}. But since $\sigma \lV$ is a positive perturbation, one can say more. Let $\sigma_1 = (d-\rho_g)/2-\epsilon$ for some small $\epsilon>0$, one has
$$\Sigma (A_{\lH_{\sigma_1}}) \subset [-\rho_g-\sigma_1,\rho_g+\sigma_1] \bigsqcup \{\|A_{\lH_{\sigma_1}}\|\}.$$
But $\|A_{\lH_{\sigma_1}}\| \geq \| A_{\lG}\| \geq d$. Thus, there is still a spectral gap and one can reproduce the perturbation analysis by considering
$$A_{\lH_{\sigma}}=A_{\lH_{\sigma_1}}+(\sigma-\sigma_1)A_{\lV}$$
and contour integral $\gamma$ around $\| A_{\lH}\|$ instead of $d$. By repeating the analysis this allows to show spectral gap up to $\sigma < d-\rho_g$. The claim about URW follows by applying Theorem \eqref{thm:graphing}.
\end{proof}

\begin{proof}[Proof of Proposition \ref{prop:envlimit}] 
Given $o\in (X,\mu)$, the random walk on the graphing produces a walk on the leaf, denoted $X_n^P(o)$. Since
\[
X_{n+1}^P(o)=X_1^P\bigl(X_n^P(o)\bigr),
\]
the environmental limit satisfies the stationarity equation in distribution
\[
(G',P',x)\;=_d\;(G',P',X_1^{P'}(x)).
\]
Because degrees are uniformly bounded, $H_P^1(x)$ is bounded and therefore integrable. Hence, by weak convergence of $(G,P,x_n)$ to $(G',P',x)$ and bounded convergence of the one-step entropy functional,
\[
\lim_{n\to\infty}\frac{1}{n}H_P^n(o)
= E_{(G',P',o')\sim\mu'}\,H_{P'}^{1}(o')
= -\int_{\mathbb G_{\bullet}^{D,M}}\sum_{y\sim x} p_{xy}\log p_{xy}\,\mu'(\mathrm d(G,P,x)).
\]

Now assume that the walk on the graphing admits a $P$-stationary probability measure $\lambda\ll\mu$. For $o'\in(X,\lambda)$, stationarity gives
\[
(G,P,X_n^P(o'))=_d(G,P,o') \quad\text{for all }n.
\]
Since the environmental limit exists and is independent of the starting point, we also have
\[
(G,P,X_n^P(o'))\;\Rightarrow\;(G',P',x).
\]
Therefore $(G',P',x)=_d(G,P,o')$.
\end{proof}

\section{Continuity of the MERW along expander sequences}

In this section we prove Theorem \ref{thm:mambrane}. 

So far, we have worked on extending the theory of MERW from finite graphs to random rooted graphs and graphings. We have shown that the URW is the natural limiting object: under mild assumptions its rooted entropy exists, equals the topological entropy, and this property uniquely characterises it among stationary random walks. In this final section, we reverse the direction of analysis. Given a sequence of finite graphs converging locally to an infinite graph for which the URW is understood, what can be inferred about the MERW on the finite graphs? Our first result is a continuity theorem showing that, under a uniform spectral gap, the MERW converges in the Benjamini–Schramm sense to the URW of the limit.

\begin{proof}[Proof of Theorem \ref{thm:mambrane}]
After breaking symmetries using i.i.d.\ labels (Bernoulli graphings), we obtain
\[
\lH_n = \lG_n + \sigma \lV_n = (X_n,\mu_n,A_{G_n+\sigma V_{\omega_n}}), \qquad
\lH = \lG + \sigma \lV = (X,\mu,A_{G+\sigma V_{\omega}}),
\]
and we denote by $o_n$ and $o$ the random roots sampled according to $\mu_n$ and $\mu$, respectively.

Consider the spectral measure of $A_{\lH_n}$ with respect to the constant function $1_{X_n}$. It is defined by
\[
\langle 1_{X_n}, A_{\lH_n}^k 1_{X_n}\rangle
= \int \lambda^k \,\nu_{n}^{1_{X_n}}(\mathrm{d}\lambda).
\]
Benjamini--Schramm convergence implies the weak convergence of the measures $\nu_{n}^{1_{X_n}}$ to $\nu^{1_X}$, the corresponding spectral measure on the graphing $\lH$. Convergence of the unique atom of these measures on $[d,d+\sigma]$ implies that the operator norms $\rho_n \to \rho$ converge. If $\rho_n = d$, then by Theorem \ref{thm:staple}, the MERW converges to SRW on $(G,o)$, and $V_n \to 0$ in the Benjamini--Schramm sense. Otherwise, if $\rho_n \to \rho > d$, choose $\delta>0$ small enough so that
\[
\lambda_2(G_n) + \sigma < d - \delta
\quad\text{and}\quad
-\rho_n + \delta \le d \le \lambda_{\min}(A_{G+\sigma V}).
\]
As usual, we realize the finite graphs and their limit as graphings, and therefore the limiting graphing also has a spectral gap. By Theorem \ref{thm:graphing}, the limiting graphing admits a URW.

It remains to show that the top eigenvectors $F_n$ converge to $F$ in the Benjamini--Schramm sense. From the convergence of the mass at the atom, we obtain
\[
\langle F_n, 1_{X_n}\rangle \to \langle F, 1_X\rangle.
\]
By the spectral decomposition of $A_{\lH_n}$, as in \eqref{eq:convestimate}, we have
\[
\bigl\| c_n F_n - \tfrac{W_k}{\rho_n^k}\bigr\|_{L^2(\lG_n)}
\le \left(1 - \frac{\delta}{d}\right)^k.
\]
Note that $W_k$ depends only on the $k$-neighborhood of $o_n$. A standard $\varepsilon/3$ argument then shows that $F_n \to F$ in the Benjamini--Schramm topology.
\end{proof}

Our second result demonstrates that the asymptotic behaviour of MERW is, in a weak sense, governed by the entropy rate along the sequence.

\begin{proof}[Proof of Theorem \ref{thm:staple}]
Write
\[
H_n = G_n + \sigma_n V_{\omega_n}, 
\qquad 
P_n = \mathrm{MERW}(H_n).
\]
Assume that $\lim \rho_n =d$. We must show first that  \(H_n \to (G,o)\), and then that \((H_n,P_n)\to (G,\mathrm{SRW},o)\).

\[
\rho_n 
\;\ge\; 
\frac{1}{|V(G_n)|} 
\langle \mathbf{1}_{V(G_n)}, A_{H_n}\mathbf{1}_{V(G_n)} \rangle
\;=\;
d + \sigma_n\,\frac{|\omega_n|}{|V(G_n)|}.
\]
Thus the perturbation $\sigma_n V_{\omega_n}$ vanishes in the BS sense (for weighted graphs), so \(H_n \to (G,o)\).

\smallskip
Let
\[
(H_n,P_n) \to (G,P,o)
\]
along a subsequence. We claim that \(P\) is necessarily the simple random walk, by the following rigidity argument. By construction, \((G,P,o)\) is a unimodular random rooted graph equipped with a Markov kernel, that is a Doob walk of energy \(d\). The Mass–Transport Principle gives
\[
1 
= \mathbb{E}_{(G,P,o)}\Big[\sum_{x \in V(G)} p_{o x}\Big]
= \mathbb{E}_{(G,P,o)}\Big[\sum_{x \in V(G)} p_{x o}\Big]
= \mathbb{E}_{(G,P,o)}\Big[\sum_{x \in V(G)} \frac{1}{d^{2}\, p_{o x}}\Big].
\]
By the harmonic–arithmetic mean inequality, equality forces
\[
p_{o x} = \frac{1}{d} \quad \text{for all } x\sim o.
\]
Hence \(P=\mathrm{SRW}\) almost surely. Since every subsequential limit is identical, compactness yields convergence of the full sequence.

\medskip

Assume now part 2). Any subsequential limit of \((H_n,P_n)\) must then be a Doob walk of energy \(\rho>d\), which rules out \(\mathrm{SRW}(G)\) as a possible subsequential limit.
\end{proof}

\bibliographystyle{plain}
\bibliography{biblio}

\end{document}